\documentclass[12pt,reqno]{amsart}
\usepackage{amssymb, mathrsfs}
\usepackage{latexsym}
\usepackage{amsmath}
\usepackage{bm}
\usepackage{amsthm}
\usepackage{amsfonts}
\usepackage{dsfont, upgreek}
\usepackage{mathtools}
\usepackage{epsfig}
\usepackage{amscd}
\usepackage{graphicx}
\usepackage{tikz-cd}
\usepackage{enumitem}
\setlist[enumerate]{itemsep=0.5ex}

\usepackage{color}
\usepackage{tikz}
\usetikzlibrary{patterns}

\usepackage[left=1.4in,right=1.4in,top=1.2in,bottom=1.2in]{geometry}

\usepackage{tikz}
\usetikzlibrary{decorations.markings}
\usetikzlibrary{arrows.meta}

\usepackage{extarrows}

\usepackage{adjustbox}
\usepackage{pgfplots}
\usepgfplotslibrary{colormaps}
\usepackage{slashed}

\usepackage[colorlinks=true,linkcolor=blue,citecolor=blue]{hyperref}

\usepackage[colorinlistoftodos,prependcaption,textsize=tiny]{todonotes}  

\usepackage[all,cmtip]{xy}
\theoremstyle{plain}% default 
\newtheorem{theorem}{Theorem}[section]
\newtheorem{proposition}[theorem]{Proposition}
\newtheorem{lemma}[theorem]{Lemma}
\newtheorem{corollary}[theorem]{Corollary}

\newtheorem{conjecture}[theorem]{Conjecture}

\theoremstyle{definition} 
\newtheorem{definition}[theorem]{Definition}

\newtheorem*{claim*}{Claim}

\newtheorem{remark}[theorem]{Remark}

\numberwithin{equation}{section}

\newcommand{\Bigwedge}{\mathord{\adjustbox{raise=.4ex, totalheight=.7\baselineskip}{$\bigwedge$}}}

\newcommand{\ind}{\textup{Ind}}

\newcommand{\dist}{\mathrm{dist}}

\newcommand{\R}{\mathbb{R}}
\newcommand{\tr}{\mathrm{tr}}

\newcommand{\supp}{\mathrm{supp}}

\newcommand{\Z}{\mathbb{Z}}

\newcommand{\sph}{\mathbb{S}}
\newcommand{\tor}{\mathbb{T}}

\newcommand{\Ric}{\mathrm{Ric}}

\newcommand{\Sing}{\mathcal{S}}

\newcommand{\interior}[1]{%
	{\kern0pt#1}^{\mathrm{\,o}}%
}

\makeatletter
\let\save@mathaccent\mathaccent
\newcommand*\if@single[3]{%
	\setbox0\hbox{${\mathaccent"0362{#1}}^H$}%
	\setbox2\hbox{${\mathaccent"0362{\kern0pt#1}}^H$}%
	\ifdim\ht0=\ht2 #3\else #2\fi
}
\newcommand*\rel@kern[1]{\kern#1\dimexpr\macc@kerna}
\newcommand*\overbar[1]{\@ifnextchar^{{\wide@bar{#1}{0}}}{\wide@bar{#1}{1}}}
\newcommand*\wide@bar[2]{\if@single{#1}{\wide@bar@{#1}{#2}{1}}{\wide@bar@{#1}{#2}{2}}}
\newcommand*\wide@bar@[3]{%
	\begingroup
	\def\mathaccent##1##2{%
		\let\mathaccent\save@mathaccent
		\if#32 \let\macc@nucleus\first@char \fi
		\setbox\z@\hbox{$\macc@style{\macc@nucleus}_{}$}%
		\setbox\tw@\hbox{$\macc@style{\macc@nucleus}{}_{}$}%
		\dimen@\wd\tw@
		\advance\dimen@-\wd\z@
		\divide\dimen@ 3
		\@tempdima\wd\tw@
		\advance\@tempdima-\scriptspace
		\divide\@tempdima 10
		\advance\dimen@-\@tempdima
		\ifdim\dimen@>\z@ \dimen@0pt\fi
		\rel@kern{0.6}\kern-\dimen@
		\if#31
		\overline{\rel@kern{-0.6}\kern\dimen@\macc@nucleus\rel@kern{0.4}\kern\dimen@}%
		\advance\dimen@0.4\dimexpr\macc@kerna
		\let\final@kern#2%
		\ifdim\dimen@<\z@ \let\final@kern1\fi
		\if\final@kern1 \kern-\dimen@\fi
		\else
		\overline{\rel@kern{-0.6}\kern\dimen@#1}%
		\fi
	}%
	\macc@depth\@ne
	\let\math@bgroup\@empty \let\math@egroup\macc@set@skewchar
	\mathsurround\z@ \frozen@everymath{\mathgroup\macc@group\relax}%
	\macc@set@skewchar\relax
	\let\mathaccentV\macc@nested@a
	\if#31
	\macc@nested@a\relax111{#1}%
	\else
	\def\gobble@till@marker##1\endmarker{}%
	\futurelet\first@char\gobble@till@marker#1\endmarker
	\ifcat\noexpand\first@char A\else
	\def\first@char{}%
	\fi
	\macc@nested@a\relax111{\first@char}%
	\fi
	\endgroup
}
\makeatother
\begin{document}

\title{$L^\infty$-metrics on tori and Schoen's conjecture}

\author{Jian Wang}
\address[Jian Wang]{Institute of Mathematics, Chinese Academy of Sciences}
\email{jian.wang.4@amss.ac.cn}
\thanks{}
\author{Jinmin Wang}
\address[Jinmin Wang]{Institute of Mathematics, Chinese Academy of Sciences}
\email{jinmin@amss.ac.cn}
\thanks{The second author is partially supported by NSFC 12501169.}
\author{Zhizhang Xie}
\address[Zhizhang Xie]{Texas A\&M University}
\email{xie@tamu.edu}
\thanks{}

\date{}
\begin{abstract}
We prove Schoen’s conjecture on $L^\infty$-metrics for tori. More precisely, we show that any $L^\infty$-metric on a torus that is smooth and has non-negative scalar curvature away from a embedded submanifold with codimension at least three extends to a smooth flat metric. We also prove a LLarull-type theorem for $L^\infty$-metric. Our proof uses weighted scalar curvature and the relative index theorem.
\end{abstract}
\maketitle

\section{Introduction}

Scalar curvature is a fundamental local invariant of a Riemannian metric, measuring the leading-order deviation of the volume of small geodesic balls from that of Euclidean spaces. Despite its local nature, lower bounds on scalar curvature can impose striking global restrictions on the topology and geometry of the underlying manifold. A prototypical manifestation of this phenomenon is provided by the torus. Geroch conjectured that $\mathbb T^n$ admits no Riemannian metric of positive scalar curvature, and this was proved  by Schoen--Yau \cite{SchoenYau79} and Gromov--Lawson \cite{GromovLawson80}. Their work also established the corresponding rigidity statement: every smooth Riemannian metric of nonnegative scalar curvature on $\mathbb T^n$ is flat.

The preceding obstruction and rigidity results naturally raise the question of how robust these phenomena are when the metric is allowed to have a small singular set. Affirmative results in this direction were obtained by Li--Mantoulidis for metrics with skeleton singularities \cite{LiMantoulidis19} and by Lee--Tam for continuous metrics \cite{LeeTam21}, under suitable hypotheses. Their work suggests that sufficiently mild singularities should not fundamentally alter the classical obstruction and rigidity picture for scalar curvature. This leads to two closely related questions: whether the obstruction to positive scalar curvature persists when the curvature condition is imposed only on the regular part, and, in the rigidity case, whether the singularities are removable. One of the main goals of this paper is to address these questions for tori and  spheres.

Throughout this paper, an $L^\infty$-metric means a measurable Riemannian metric uniformly equivalent to a smooth background metric $g_0$: there exists $\Lambda\geq 1$ such that
\[
  \Lambda^{-1}g_0\leq g\leq\Lambda g_0
  \qquad\text{almost everywhere}.
\]
We assume that $g$ is smooth on the complement of a closed singular set $\mathcal S$, and impose scalar curvature bounds only on the regular part. In this setting, Schoen proposed the following conjecture; see \cite[Conjecture 1.5]{LiMantoulidis19}.

\begin{conjecture}[Schoen]\label{conj:Schoen}
Let $M^n$ be a closed manifold with nonpositive Yamabe invariant. Suppose that $g$ is an $L^\infty$-metric, smooth away from a  embedded submanifold  $\mathcal S$ of codimension at least three. If the scalar curvature $R(g)$ of $g$ satisfies $R(g)\geq 0$ on $M\setminus \mathcal S$, then $g$ extends to a smooth Ricci-flat metric across $\mathcal S$, allowing a change of smooth structure along $\mathcal S$.
\end{conjecture}

Here and below, smooth extension allows a change of coordinates across the singular set; this qualification is necessary even for pullbacks of flat metrics \cite[Remark 1.4]{CecchiniFrenckZeidler24}. Our main result establishes the torus case. 

\begin{theorem}
\label{A}
Conjecture~\ref{conj:Schoen} holds when $M=\mathbb T^n$. 
\end{theorem}

In fact, in our Theorem \ref{A}, the singular set $\Sing$ is allowed to be an embedded simplicial complex. The metric and dimensional assumptions in Conjecture \ref{conj:Schoen} are essentially sharp. Without a uniform lower bound on the metric, constructions based on the negative-mass Schwarzschild metric give counterexamples on tori. In codimension two, on the other hand, cone angles greater than $2\pi$ can obstruct removability of singularity even for metrics that are uniformly equivalent to the Euclidean metric \cite[Sections 8.2--8.3]{LiMantoulidis19}. Thus, one cannot allow either arbitrary degeneration of the metric or  codimension-two singularities. From a technical standpoint, the assumption on codimension being at least three guarantees that $\mathcal S$ has zero $W^{1,2}$-capacity, allowing weak elliptic arguments to pass across the singular set.

We remark that Schoen’s conjecture \ref{conj:Schoen} is false in full generality. Cecchini--Frenck--Zeidler constructed counterexamples to Schoen's conjecture on a class of simply connected spin manifolds \cite{CecchiniFrenckZeidler24}. Nevertheless, it remains a very interesting problem to determine under what additional assumptions the conjecture does hold, possibly  subject to further topological constraints on the underlying manifold \(M\). Several positive results are known. In dimension three, Conjecture~\ref{conj:Schoen} was established by Li--Mantoulidis \cite{LiMantoulidis19}, while in dimension four Kazaras proved it for tolerable singular sets $S$ \cite{Kazaras24}. Llarull-type rigidity theorems for $L^\infty$-metrics with certain singular sets have also been obtained by Chu--Lee--Zhu \cite{ChuLeeZhu24} and Wang--Xie \cite{WangXie24}. For the special class of conical metrics, scalar curvature rigidity has been established by Dai--Sun--Wang \cite{DaiSunWang24,DaiSunWang25spin} and Jovanovic--Wang \cite{JovanovicWang25}.

For tori, Wang--Xie \cite{WangXie24} partially confirmed Conjecture \ref{conj:Schoen} by proving that if $\operatorname{codim}(\mathcal S)\ge \frac{n}{2}+1$ and the induced map $i_*\colon \pi_1(\mathcal S)\to \pi_1(\mathbb T^n)$ is trivial, then any $L^\infty$-metric on $\mathbb T^n\setminus \mathcal S$ satisfying $R(g)\geq 0$ must be Ricci-flat. This result was subsequently strengthened by Dai--Wang--Wang--Wei \cite{DaiWangWangWei24}, who showed that under the codimension assumption $\operatorname{codim}(\mathcal S)\geq \max\left\{4,\frac{n}{2}+1 \right\}$, the resulting Ricci-flat metric extends smoothly across $\mathcal S$ to a flat metric on the entire torus.

In the asymptotically flat setting, Lee--LeFloch proved a positive mass theorem with distributional scalar curvature \cite{LeeLeFloch15}, and Dai--Sun--Wang treated isolated conical singularities \cite{DaiSunWang24Mass}. More recently, Khuri--Wang--Wang established positive mass in all dimensions for asymptotically flat $L^\infty$-metrics with subcritical singular sets, together with Euclidean rigidity under a stronger dimension bound \cite{KhuriWangWang26}. These developments, together with the corresponding results in the torus and spherical settings, identify these three settings as natural domains for scalar curvature rigidity in the presence of singularities.

A main new ingredient of this paper is a conformal blow-up that preserves positivity of a weighted scalar curvature. After a preliminary deformation, it suffices to rule out metrics with uniformly positive scalar curvature on the regular part. We make the regular part complete by conformally blowing up the metric near $\mathcal S$, while controlling the following weighted scalar curvature 
\[
  R_{\mu,f}(g)
  =R(g)+2\Delta_g f-\mu|\nabla f|_g^2,
  \qquad \mu>\frac{n-1}{n}.
\]
The weight $f$ compensates for the curvature terms introduced by the conformal factor. The flexibility in choosing $\mu>(n-1)/n$  allows a blow-up strong enough to produce completeness in codimension three, while preserving the index obstruction. The case $\mu=1$ is the $P$-scalar curvature studied in related work, including Zhou--Zhu's weighted Llarull theorem \cite{ZhouZhu25}.

Our approach in fact applies to general closed singular sets, rather than only to simplicial subsets or submanifolds, provided that their Hausdorff dimension is sufficiently small. For such general singular sets, however, we impose the additional topological assumption that there exists a neighborhood $U$ of $\mathcal S\subset \mathbb T^n$ such that
\begin{equation}\label{eq:nonsurjective}
    \pi_1(U)\longrightarrow \pi_1(\mathbb T^n) \textup{ is not surjective. }
\end{equation}
 This condition allows us to construct two coverings of $\mathbb T^n$ that agree near $\mathcal S$ but carry different global index theoretic data. The relative index theorem then gives a nonzero higher index class, whereas positivity of the weighted scalar curvature forces the corresponding Callias-type operator to be invertible, hence a zero index class. The resulting contradiction gives the desired obstruction and, together with the preliminary deformation, implies Ricci-flatness on the regular part; see Theorem~\ref{B}. This index-theoretic argument builds on the methods of \cite{XieYu14,Wang25}.

For \emph{simplicial singular sets}, the above topological assumption \eqref{eq:nonsurjective} can be removed by refining the conformal blow-up. When $\dim \mathcal S\leq n-3$, we construct a complete conformal metric $g'=u^{2a}g$ on the regular part satisfying
\[
  R_{\mu,f}(g')\geq C u^{-2a}d_g(\,\cdot\,,\mathcal S)^{-2}
\]
near $\mathcal S$, while keeping the curvature error arbitrarily small away from $\mathcal S$. This lower bound absorbs the errors introduced when deforming a map to the sphere so that it becomes constant near $\mathcal S$. It allows us to prove the following extension of Llarull's scalar curvature rigidity theorem for spheres \cite{Llarull98} to the setting of $L^\infty$-metrics.

\begin{theorem}\label{C}
Let $M^n$ be a closed connected spin manifold with $n\geq3$, and let $g$ be an $L^\infty$-metric smooth outside a finite Lipschitz embedded simplicial subcomplex $\mathcal S$ of dimension at most $n-3$. Suppose there is a Lipschitz map
\[
  \Phi:(M,g)\longrightarrow(\mathbb S^n,g_{\mathbb S^n})
\]
of nonzero degree, smooth away from $\mathcal S$, and area-non-increasing. If $R(g)\geq n(n-1)$ on $M\setminus \mathcal S$, then $\Phi$ is an isometry of metric spaces.
\end{theorem}

For a closed spin enlargeable manifold, the same strategy, applied to the contracting maps arising from enlargeability, shows that any $L^\infty$-metric that is smooth away from a simplicial singular set of codimension at least three and has nonnegative scalar curvature on the regular part must be Ricci-flat there. In particular, any such $L^\infty$-metric on a torus $\mathbb T^n$ is Ricci-flat away from the singular set.

To complete our proof of Schoen's conjecture for tori, we prove that if a Ricci-flat metric is smooth away from a singular set $\mathcal S$ of codimension at least three, or more generally if its Minkowski dimension $\dim_{\mathcal M}\mathcal S\leq n-3+\frac{1}{(n-1)}$, then it extends across $\mathcal S$ to a smooth Ricci-flat, in fact flat, metric on $\mathbb T^n$. Dai--Wang--Wang--Wei \cite{DaiWangWangWei24} established the corresponding conclusion in codimension at least four using the $\mathrm{RCD}(0,n)$ structure of the singular space. We treat the remaining codimension-three case by constructing $n$ harmonic maps to $\mathbb S^1$ representing independent cohomology classes. De Giorgi--Nash--Moser theory, together with gradient estimates on the regular part, controls their behavior near $\mathcal S$. The Bochner formula, combined with a refined Kato inequality, then forces their Hessians to vanish. Their differentials therefore form a parallel frame of the cotangent bundle, which gives a global identification with a flat torus, hence the required smooth extension across $\mathcal S$.

The paper is organized as follows. In Section 2 we establish the elliptic and conformal deformation results for \(L^\infty\)-metrics. In Section 3 we prove the weighted scalar curvature obstructions for punctured tori and the weighted Llarull theorem using index theory. Section 4 develops the conformal blow-up construction, with quantitative curvature estimates near simplicial singular sets. Section 5 combines these tools to establish the \(L^\infty\) Llarull theorem and the scalar curvature obstructions for tori and spin enlargeable manifolds. Section 6 completes the torus rigidity argument by constructing harmonic coordinates and proving removability of the singular set. The appendix contains the potential estimates used in Section 4.

    \section{Conformal change for $L^\infty$-metric}
In this section we develop the elliptic theory needed to deform
$L^\infty$-metrics whose singular sets have codimension greater than two.
We apply this theory to obtain a scalar-curvature dichotomy and an analogous
statement in the Llarull setting.

\begin{definition}\label{linfinity}
Let $M^n$ be a smooth manifold equipped with a smooth Riemannian
metric $g_0$. A measurable symmetric $(0,2)$-tensor $g$ is called an \textup{$L^\infty$-metric} with respect to $g_0$ if every point $x\in M^n$ admits a neighborhood $V_x$ and
a constant $\Lambda_x\geq 1$ such that
\begin{equation}\label{Lambda}
\Lambda_x^{-1} g_0 \leq g \leq \Lambda_x g_0
\qquad \text{a.e. on } V_x
\end{equation}
as quadratic forms. We say that $(M^n,g)$ is a \textup{complete $L^{\infty}$-Riemannian manifold} if $g_0$ is complete and
\eqref{Lambda} holds with a uniform constant $\Lambda\geq 1$. The \textup{regular set} of an $L^\infty$-metric $g$ consists of those points having a neighborhood on which $g$ is smooth (i.e. $g$ admits a smooth representative), and the \textup{singular set}
is the complement of the regular set.
\end{definition}

The main result of this section is the following dichotomy for non-negative scalar curvature, which generalizes a well-known result of Kazdan in the smooth setting \cite{kazdan}.

\begin{theorem}\label{dichotomy}Let $(M^n,g)$ be a closed
$L^\infty$-manifold. Assume that the singular set $\mathcal{S}$ satisfies the Minkowski dimension upper bound $\dim_{\mathcal M}\mathcal{S}< n-2$.
Suppose that $R(g)\geq 0$ on $M^n\setminus \mathcal{S}$. Then the following dichotomy holds true:
\begin{itemize}
\item[(1)] either the Ricci curvature vanishes on $M^n\setminus \mathcal{S}$,

\item[(2)] or there is a $L^\infty$-metric $g'$ on $M^n$ such that $g'$ is smooth on $M^n\setminus \mathcal{S}$ and its scalar curvature $R(g')$ is bounded blow by a positive constant on $M^n\setminus \mathcal{S}$. 
\end{itemize}
\end{theorem}

We also use the following Llarull type  dichotomy. An analogous result also appeared in \cite{WangXie24}.  

\begin{theorem}\label{dichotomy-Llarull} Let $(M^n, g)$ be a closed $L^\infty$-manifold. Assume that the singular set $\mathcal{S}$ satisfies the Minkowski dimension upper bound $\dim_{\mathcal{M}}\mathcal{S}< n-2$. Suppose that $\Phi: (M, g)\rightarrow (\mathbb{S}^n, g_{\mathbb{S}^n})$ satisfies that on $M\setminus \mathcal{S}$
\begin{equation*}
R(g)\geq n(n-1) |\Bigwedge^2 d\Phi|_g.
\end{equation*}
Then the following dichotomy holds true: 
\begin{itemize}
    \item[(1)] $\Phi$ is an isometry, 
    \item[(2)] or there is a $L^\infty$-metric $g'$ on $M$ such that $g'$ is smooth on $M\setminus \mathcal{S}$ and $R(g')-n(n-1)|\Bigwedge^2d\Phi|_{g'}$ is bounded below by a positive constant on $M^n\setminus \mathcal{S}$
\end{itemize}
\end{theorem}

\subsection{Existence of positive solution} When the singular set $\mathcal{S}$ is of Minkowski codimension greater than two, it has vanishing $W^{1,2}$-capacity. In particular,
it is invisible to $W^{1,2}$-analysis in the sense that $W^{1,2}(M^n\setminus\mathcal{S},g)$ agrees with $W^{1,2}(M^n,g)$, and integration by parts remains valid across $\mathcal S$. More precisely, we have the following result. 

\begin{proposition}\cite[Proposition 2.8]{Khuri-Wang-Wang}\label{int-by-part}
Let $(M^n,g)$ be a compact  $L^\infty$-manifold.
If the singular set $\mathcal{S}$ is compact and satisfies the Minkowski dimension upper bound
$\dim_{\mathcal M} \mathcal S < n-2$,
then the singular set has vanishing \(W^{1,2}\)-capacity with respect to \(g\).
In particular consider \(u\in W^{1,2}(M^n,g)\), and let
$v\in C^\infty(M^n\setminus \mathcal S)$ in addition to
\begin{equation}
v\in W^{1,2}(M^n,g),\qquad
|\nabla v|_g\in L^\infty(M^n\setminus \mathcal S),
\qquad
\Delta_g v\in L^2(M^n\setminus \mathcal S),
\end{equation}
then
\begin{equation}
\int_{M^n\setminus \mathcal S} u\,\Delta_g v\,dV_g
=
-\int_{M^n\setminus \mathcal S}
\langle \nabla u,\nabla v\rangle_g\,dV_g .
\end{equation}
\end{proposition}

\begin{remark} Proposition \ref{int-by-part} remains valid if the Minkowski dimension assumption is replaced by the Hausdorff dimension assumption $\dim_{\mathcal{H}}\mathcal S< n-2$. 

Consequently, all subsequent arguments in this section that rely on the dimension assumption only through Proposition~\ref{int-by-part} remain valid under the Hausdorff dimension hypothesis. In particular, Theorems~\ref{dichotomy} and \ref{dichotomy-Llarull} also
hold for $\dim_{\mathcal H}\mathcal S<n-2 .$  
\end{remark}

\vspace{3mm}

Consider the eigenvalue and the eigenfunction problem
\begin{equation}
    \lambda_1:=\inf\left\{\int_{M^n\setminus \mathcal{S}} |\nabla u|^2+ \textbf{c} u^2  ~{\big|}~ ~u\in W^{1,2}(M, g)~\text{ and }\int_{M\setminus \mathcal{S} }u^2=1\right\} 
\end{equation}
where $\textbf{c}$ is a smooth function compactly supported in $M\setminus \mathcal{S}$. The following result shows the existence and the positivity of the first eigenfunction. 

\begin{proposition}\label{positive-solu} Let $(M^n, g)$ be a closed $L^\infty$-manifold. Assume that the singular set $\mathcal{S}$ satisfies the Minkowski dimension upper bound $\dim_{\mathcal{M}}\mathcal{S}< n-2$. Suppose that $\textbf{c}\in C^\infty_{c}(M\setminus \mathcal{S})$. Then there is a positive function 
\[ u\in W^{1,2 }(M^n, g)\cap C^\infty(M^n\setminus \mathcal{S})\cap L^\infty(M) \] satisfying 
\begin{equation}\label{eignevalue-weak}
    -\Delta_g u+ \textbf{c}u=\lambda_1 u
\end{equation}and $u$ is bounded below by a positive constant in the sense of essential-infimum sense, that is 
$
    \operatorname*{ess\,inf}_{M^n} (u) >0 .
$
\end{proposition}
\begin{proof}Since $g$ is uniformly equivalent to a smooth background metric, it follows from \cite[Theorem 8.38]{GT-Book} and the regularity theory \cite[Corollary 8.22]{GT-Book} that there exists a non-negative solution $u\in W^{1,2 }(M^n, g)\cap L^\infty(M)\cap C^\infty(M\setminus \mathcal{S})$ to \eqref{eignevalue-weak}. The weak Harnack inequality \cite[Theorem 8.18]{GT-Book} then shows that $u$ must be strictly positive away from the singular set. Hence \(u\) has a representative
which is strictly positive on \(M^n\), and
$\operatorname*{ess\,inf}_{M}u>0
$.
\end{proof}

\subsection{Proof of Theorems \ref{dichotomy} and \ref{dichotomy-Llarull}}

We now prove Theorem \ref{dichotomy}
\begin{proof}[The proof of Theorem \ref{dichotomy}] We first establish a weaker version of Theorem \ref{dichotomy} for scalar curvature. Suppose that there is a point $p\in M^n\setminus \mathcal{S}$ with $R(g)>0$. Then there is a small geodesic ball $B(p, 2r)\Subset M^n\setminus \mathcal{S}$ on which scalar curvature is strictly positive. Choose a cut-off function $\varphi\in C^\infty_c(B(p, 2r))$ with $\varphi=1$ on $B(p, r)$, and $0\leq\varphi\leq 1 $. Consider the following equation
\begin{equation}
    -\Delta_g u+\frac{n-2}{4(n-1)}\varphi R(g) u=\lambda_1 u
\end{equation}
where $\lambda_1>0$ is the first eigenvalue of $-\Delta+\frac{n-2}{4(n-1)}\varphi R(g)$. Proposition \ref{positive-solu} provides a positive solution $u\in C^\infty(M^n\setminus \mathcal{S})\cap W^{1,2}(M^n, g)\cap L^\infty (M^n)$. Moreover, $u$ is uniformly bounded below by a positive constant. Then the conformal change $g'=u^{\frac{4}{n-2}}g$ yields a $L^\infty$-manifold $(M, g')$ satisfying that 
\begin{equation}
    R(g')=u^{-\frac{4}{n-2}}\left( \lambda_1+(1-\varphi)R(g)\right)\geq \lambda_1(\sup_{x\in M} u)^{-\frac{4}{n-2}}>0. 
\end{equation}
We conclude that if the scalar curvature $R(g)$ is strictly positive at some point away from $\mathcal{S}$, then $g'$ is a candidate required  in item $(2)$ of Theorem \ref{dichotomy}. 

Now suppose that $R(g) = 0$ on $M^n\setminus \mathcal{S}$. Assume that the Ricci curvature  $\operatorname{Ric}_g(p)\neq 0$ for some $p\in M^n\setminus \mathcal{S}$. Following the classical strategy of Schoen--Yau \cite{MR526976}, we shall show that there exists an $L^\infty$-metric $g'$ on $M^n$, smooth on $M^n\setminus \mathcal{S}$, whose scalar curvature is bounded below by a positive constant on $M^n\setminus \mathcal{S}$.  Choose a geodesic ball
\(B_{2r}(p)\Subset M^n\setminus\mathcal S\), and a cut-off function
\(\varphi\in C_c^\infty(B_{2r}(p))\) as above. For
\(t>0\) sufficiently small, define
\begin{equation}
        g_t:=g-t\varphi\,\operatorname{Ric}_g .
\end{equation}
Then \(g_t=g\) outside \(B_{2r}(p)\).Recall the linearization of the scalar curvature 
\begin{equation*}
        \left.\frac{d}{dt}\right|_{t=0}R(g)
        =
        -D R(g)(\varphi\operatorname{Ric}_g),
\quad 
        D_hR(g)
        =
        -\Delta_g(\operatorname{tr}_g h)
        +\operatorname{div}_g\operatorname{div}_g h
        -\langle \operatorname{Ric}_g,h\rangle_g .
\end{equation*}
Using the assumption that \(R(g)=0\) and integrating by parts gives 
\begin{equation}\label{tscalar}
        \int_{M^n} R(g_t)\,dV_{g_t}
        =
        t\int_{M^n} \varphi |\operatorname{Ric}_g|^2\,dV_g
        +O(t^2)>0
\end{equation}
for all sufficiently small \(t>0\). Now solve
\begin{equation}
    -\Delta_{g_t} u_t+\frac{n-2}{4(n-1)}R(g_t) u_t=\lambda_1(t) u_t
\end{equation}
It follows from \eqref{tscalar} that $\lambda_1(t)>0$. The same conformal change argument as above gives  a $L^\infty$-metric $g'_t$,  smooth on $M^n\setminus \mathcal{S}$, whose scalar curvature is bounded below by a positive constant on $M^n\setminus \mathcal{S}$.

Consequently, the only remaining possibility is that $\operatorname{Ric}_g=0$ on $M^n\setminus\mathcal{S}$. This completes the proof of the theorem.
\end{proof}

\begin{proof}[Proof of Theorem \ref{dichotomy-Llarull}]
We first establish a weaker version of Theorem \ref{dichotomy-Llarull}. 
Suppose that there is a point $p\in M^n\setminus \mathcal{S}$ with \[R(g)_p-n(n-1)|\Bigwedge^2d_p\Phi|_g>0.\]Then there is a small geodesic ball $B(p, 2r)\Subset M^n\setminus \mathcal{S}$ on which the left hand side of the above inequality is bounded below by a positive constant. Choose a cut-off function $\varphi\in C^\infty_c(B(p, 2r))$ with $\varphi=1$ on $B(p, r)$, and $0\leq\varphi\leq 1 $. Consider the following equation
\begin{equation}
    -\Delta_g u+ \frac{n-2}{4(n-1)}\varphi(R(g)-n(n-1)|\Bigwedge^2d\Phi|_g)u=\lambda_1 u
\end{equation} where $\lambda_1>0$ is the first eigenvalue. Using the positive solution from Proposition \ref{positive-solu} and the conformal change yields a $L^\infty$-metric $g'$ on $M^n$ satisfying that 
\begin{equation}
    R(g')-n(n-1)|\Bigwedge^2d\Phi|_{g'}>C>0.
\end{equation}

We consider the remaining case where $R(g)=n(n-1)|\Bigwedge^2d\Phi|_g$. Now the same argument of \cite[Lemma 2.1]{WangXie24} implies that,  unless$\Phi$ is an isometry, one can perform a conformal change to obtain a $L^\infty$-metric $g'$ on $M$ such that $g'$ is smooth on $M\setminus \mathcal{S}$ and $R(g')-n(n-1)|\Bigwedge^2d\Phi|_{g'}$ is bounded below by a positive constant on $M^n\setminus \mathcal{S}$. This finishes the proof. 
\end{proof}

    \section{$(\mu,f)$-Weighted scalar curvature}\label{sec:weightedSc}

Classical scalar-curvature extremality results were established for tori by Schoen-Yau   \cite{SchoenYau79} and Gromov-Lawson  \cite{GromovLawson80}, and  for spheres by Llarull 
\cite{Llarull98}. In this section, we prove analogous results for weighted
scalar curvature.

Following Perelman, we consider a family of weighted scalar curvatures associated with a smooth function $f$ and a positive parameter $\mu$:
\[
R_{\mu,f}(g)=R(g)+2\Delta f-\mu|\nabla f|^2.
\]
This quantity has also been studied in \cite{Fan08,Deng21,ZhouZhu25} and is
sometimes referred to as the \(P\)-scalar curvature when $\mu=1$. Throughout this section,
we focus on the range
$
\mu>\frac{n-1}{n}.
$
	\begin{definition}
		Let $(M,g)$ be a smooth Riemannian manifold. For $\mu>0$ and $f\in C^\infty(M)$, we define the $(\mu,f)$-weighted scalar curvature as
		$$R_{\mu,f}(g)=R(g)+2\Delta f-\mu|\nabla f|^2.$$
        %For our purposes, we focus on $\mu>\frac{n-1}{n}$.
	\end{definition}

We first establish an extremality result for punctured tori.

\begin{theorem}\label{thm:torus}
	Let $\mathcal{S}\subset \tor^n$ be a closed set satisfying that there is a neighborhood $U$ of $\mathcal{S}$ such that  
	\begin{equation}\label{fundament-group}i_*\colon \pi_1(U)\to \pi_1(\tor^n) \textup{ is not surjective.}\end{equation}
	Then $\tor^n\setminus \mathcal{S}$ does not admit any complete metric with positive $(\mu,f)$-weighted scalar curvature for any $\mu>\frac{n-1}{n}$ and $f\in C^\infty(M)$. 
\end{theorem}

The range \(\mu>\frac{n-1}{n}\) gives rise to the optimal dimensional bound for the singular set in Themorem \ref{B}. Extending the argument from the \(P\)-scalar curvature
case \(\mu=1\) to this full range requires a refined estimate for the modified
Dirac operator.  We note that the condition \eqref{fundament-group} in Theorem \ref{thm:torus} cannot be dropped, see 

For comparison, Brendle-Wang \cite{BrendleWang26} studies the werighted scalar curvature in an integral sense and used the distinct critical value  $\mu=\frac{n+1}{n+2}$.

We also establish an extremality result extending Llarull's theorem to open spin manifolds.

\begin{theorem}\label{thm:Llarull}
Let $M^n$ be an open spin manifold, and $\Phi\colon M^n\to \sph^n$ be a smooth map satisfying that $\Phi$ is locally constant outside some compact set $K$ and the degree of $\Phi $ is non-zero. Then there is no complete metric $g$ on $M$ with
\[R_{\mu, f}(g)-2|\wedge^2 d\Phi|_{\tr}>0,\]
where $\mu>\frac{n-1}{n}$, and $f\in C^\infty(M)$ and $|\wedge^2 d\Phi|_{\tr}:=\sqrt{\tr\left((\wedge^2 d\Phi_x)^*(\wedge^2 d\Phi_x)\right)},$
\end{theorem}
\subsection{The tours case}
The purpose of this subsection is  to prove  Theorem \ref{thm:torus} using  the $\mathbb{Z}^n$-equivariant index. To this end, we will construct two
coverings of
$$
M:=\tor^n\setminus\mathcal S
$$
which agree near the lift of $\mathcal{S}$ but differ globally. The
relative index theorem then yields a nonzero class in
$K_*(C^*\mathbb{Z}^n)$. On the other hand, positivity of the
$(\mu,f)$-weighted scalar curvature forces this class to vanish, yielding
the desired contradiction.

\subsubsection{Two $\mathbb{Z}^n$ coverings} \label{subsubsec:covering} We begin with the universal cover   $\widetilde{\tor}^n$ of $\tor^n$ and
the lift $\widetilde{\mathcal S}\subset \widetilde{\tor}^n$ of $\mathcal S$.  Then
\[ 
\widetilde M:=\widetilde{\tor}^n\setminus\widetilde{\mathcal S}\rightarrow M
\] is a $\mathbb{Z}^n$-covering.

For another covering, we consider the proper subgroup of $\Gamma:=i_*(\pi_1(U))\subset \mathbb{Z}^n$. Hence, there is a covering 
$\tor^{n,\Gamma}$ of $\tor^n $ corresponding to $\Gamma$ and the lift  
$\mathcal S^\Gamma\subset\tor^{n,\Gamma} $ of
$\mathcal S$. Set 
\[
M^\Gamma:=\tor^{n,\Gamma}\setminus\mathcal S^\Gamma.
\] By the construction, the lift $\mathcal S^\Gamma$ is connected. Moreover, the  universal property of $\widetilde{\tor}^n$ gives  a cover map
\begin{equation}\label{required-cover}
    \widetilde{M}=\widetilde{\tor}^n\setminus \widetilde{\mathcal{S}}\longrightarrow M^\Gamma:=\tor^{n, \Gamma}\setminus\mathcal{S}^\Gamma
\end{equation} In particular,  $\widetilde{\mathcal S}$ is the disjoint union of
$\lvert\mathbb{Z}^n/\Gamma\rvert$ copies of
$\mathcal S^\Gamma$.

\vspace{2mm}

The cover $M^\Gamma$ carries a $\Gamma$-action, rather than  $\mathbb{Z}^n$-action. To obtain $\mathbb{Z}^n$-invariant, choose a set $I\subset \mathbb{Z}^n$ of representatives for the left coset of $\Gamma$, i.e 
$
\mathbb Z^n=\bigsqcup_{a\in I}a\Gamma
$.  Define that 
\begin{equation}\label{eq:hat-torus}
\widehat{\tor}^n
:=
I\times\tor^{n,\Gamma}
 \quad \quad  \widehat{M}:=I\times M^\Gamma \quad \widehat{\mathcal{S}}:=I \times \mathcal{S}^\Gamma.
\end{equation} Since $\Gamma$ is a proper subgroup of $\mathbb{Z}^n$, the set $I$
has at least two elements; in particular, $\widehat M$ is disconnected. 

\vspace{2mm}

We now define a $\mathbb{Z}^n$-action on $\widehat{\tor}^n$, and hence
on $\widehat M$. For $z\in\mathbb{Z}^n$ and $a\in I$, there is a
unique pair $(b,\gamma)\in I\times\Gamma$ with
$
za=b\gamma.
$
Define
\begin{equation}\label{eq:Zn-action}
z\cdot(a,p):=(b,\gamma^{-1}\cdot p),
\qquad
p\in\tor^{n,\Gamma}.
\end{equation}
This is a well-defined $\mathbb{Z}^n$-action. Its restriction to
$\widehat M$ preserves $\widehat{\mathcal S}$, and therefore induces a
$\mathbb{Z}^n$-action on $\widehat M$.

\begin{remark}\label{two-action-compar} The two coverings are globally different: $\widetilde M$ is connected,
whereas $\widehat M$ is disconnected. Nevertheless, they are 
homeomorphic near the respective lifts of $\mathcal S$, and this
identification is $\mathbb{Z}^n$-equivariant.

From \eqref{required-cover}, we choose some neighborhood $U_2\subset \tor^n$ of $\mathcal{S}$ such that $\widetilde U_2\cong I\times U_2^\Gamma$, where
$\widetilde U_2$ and  $U_2^\Gamma$ are the lifts of $U_2$ in $\widetilde{\tor}^n$ and $\tor^{n,\Gamma}$. Let $\widehat U_2$ be the lift of $U_2$ in  $\widehat{\tor}^n$ and  fix a component $(U^\Gamma_2)_0$ of $\widetilde{U}_2$ containing the base point. From \eqref{required-cover}, $(U^\Gamma_2)_0$ is homoemorphic to $U^\Gamma_2$ and any element in $\widetilde{U}_2$ can be written in the form  $a\cdot p$, where $a\in I$ and $p\in (U^\Gamma_2)_0$. 
Combining it with \eqref{eq:hat-torus} yields a homeomorphism 
\begin{equation}
T\colon \widetilde U_2\to\widehat U_2, \quad T(a\cdot p)=(a,p)\in I\times U^\Gamma_2.
\end{equation} Moreover, $T$ is $\mathbb{Z}^n$-invariant. 
\end{remark}

\subsubsection{Construct non-zero element in $K_*(C^*\mathbb{Z}^n)$}\label{subsubsec:index} Consider  a complete Riemannian metric $g$ on $M=\tor^n\setminus\mathcal S$ and  choose relatively compact neighborhoods of $\mathcal{S}$ and a cut-off function $\rho$ on $M$ 
\[
 \mathcal{S}\subset U_1\Subset U_2\subset \tor^n, \quad 
\rho\equiv1\quad\text{on }U_1,
\quad
\rho\equiv0\quad\text{on }\tor^n\setminus U_2.
\]

Since $d\rho$ has compact support, its Lipschitz constant with respect to
$g$,
\begin{equation}\label{cut-off-lip-constant}
L_\rho:=\sup_M\lvert d\rho\rvert_g,
\end{equation}
is finite.

Let $\widetilde g$ and $\widehat g$ be
lifts of $g$ to $\widetilde M$ and $\widehat M$, and similarly let $\widetilde \rho$ and $\widehat \rho$ lifts of $\rho$ to $\widetilde M$ and $\widehat M$. From Remark \ref{two-action-compar}, the map $T: \widetilde{U}_2\setminus \widetilde{\mathcal{S}}\rightarrow \widehat{U}_2\setminus \widehat{\mathcal{S}}$ is a $\mathbb{Z}^n$-equivariant isometry.
Fix  a spin structure on $\tor^n$, and  denote by $S(TM)$ the spinor bundle  $M$.
Let  $S(T\widetilde M)$ and $S(T\widehat M)$ be the corresponding lifted spinor bundles to $\widetilde M$ and $\widehat M$, respectively.  The map
$T$ induces a $\mathbb{Z}^n$-equivariant unitary bundle isomorphism
\begin{equation}\label{eq:spinor-identification}
\mathcal V\colon
S(T\widetilde U_2)
\longrightarrow
S(T\widehat U_2),
\end{equation}
which is parallel and commutes with Clifford multiplications.

Consider the disjoint union $\widetilde M\sqcup \widehat M$ with the metric $\widetilde g\sqcup\widehat g$. The $L^2$-space of the spinor bundle $S(T\widetilde M)\sqcup S(T\widehat M)$ is 
	$L^2(\widetilde M,S(T\widetilde M))\oplus L^2(\widehat M,S(T\widehat M))$.
Define the Dirac operator, the connection and the Clifford action
	\begin{equation}
		\mathbf{D}=\begin{pmatrix}
			\widetilde D&0\\0&-\widehat D
		\end{pmatrix}, \quad \mathbf{\nabla}=\begin{pmatrix}
			\widetilde\nabla&0\\0&\widehat\nabla
		\end{pmatrix} \quad 
		\mathbf{c}(X)=\begin{pmatrix}
			\widetilde c(\widetilde X)&0\\0&-\widehat c(\widehat X)
		\end{pmatrix}
	\end{equation}
	for $X=\widetilde X\oplus \widehat X\in C_c^\infty(\widetilde M,T\widetilde M)\oplus C_c^\infty(\widehat M,T\widehat M)$.

   We now introduce a Callias-type perturbation. Set
	\begin{equation}
		\bm{\rho}=\begin{pmatrix}
			\widetilde\rho&0\\0&\widehat\rho
		\end{pmatrix},\qquad \mathcal U=\begin{pmatrix}
			0&\mathcal V^*\\
			\mathcal V&0
		\end{pmatrix}.
	\end{equation}
	For $\varepsilon>0$, define the Callias-type operator
	\begin{equation}
		B=D+\varepsilon \bm\rho\cdot\mathcal U
	\end{equation} on 
      $
		H_0=C_c^\infty(\widetilde M,S(T\widetilde M))\oplus C_c^\infty(\widehat M,S(T\widehat M)).
	$ We equip $H_0$ with the graph norm 
	$$\|\sigma\|_B^2=\|\sigma\|^2+\|B\sigma\|^2,$$ and let \(\operatorname{Dom}(B)\) be its completion with respect to this
norm. Such a operator is uniformly invertible outside a $\mathbb{Z}^n$-cocompact set. 
	
	\begin{lemma}
		The operator $B$ is uniformly invertible outside a $\Z^n$-cocompact set. More precisely, for any $\sigma\in H_0$ supported inside $\widehat U_1\sqcup\widetilde U_1$, we have
		$$\|B\sigma\|^2\geq\varepsilon^2\|\sigma\|^2.$$
	\end{lemma}
	\begin{proof}
		Note that on $\widehat U_1\sqcup\widetilde U_1$, we have $\overbar\rho\equiv 1$ and $\mathcal UD=-D\mathcal U$. Therefore,
		\begin{equation}
			\|B\sigma\|^2=\|D\sigma\|^2+\varepsilon^2\|\mathcal U\sigma\|^2\geq\varepsilon^2\|\sigma\|^2.
		\end{equation}
	\end{proof}
	
	We show that the index of $B$ is a non-zero element in $K_*(C^*\mathbb{Z}^n)$, using  the relative index theorem (see, for example, \cite{HochsWang26,Wang25,XieYu14}).
	\begin{lemma}\label{lemma:index}
		The operator $B$ defines an index class $\ind(B)$ in $K_*(C^*\Z^n)$, such that $\ind(B)=0$ if $B$ is bounded from below. Moreover, when $\Gamma$ is a proper subgroup of $\Z^n$, $\ind(B)\ne 0$ in $K_*(C^*\Z^n)$.
	\end{lemma}
	\begin{proof}
		Since $B$ is invertible outside a $\Z^n$-cocompact set, the index class $\ind(B)\in K_n(C^*\Z^n)$ is well-defined  (see \cite[Section 4]{Wang25} or \cite[Section 3]{XieYu14} for the precise definition).
		By the relative index theorem \cite[Theorem A]{XieYu14} and the computation in \cite[Section 4]{Wang25}, we have
		\begin{equation}\label{relative-index}
			\ind(B)=\ind(D_{\widetilde \tor^n})-\ind(D_{\widehat \tor^n})\in K_*(C^*\Z^n),
		\end{equation}
		where $D_{\widetilde \tor^n}$ and $D_{\widehat \tor^n}$ are the Dirac operators on $\widetilde \tor^n$ and $\widehat \tor^n$ with respect to any smooth metric lifted from $\tor^n$, respectively. 

        \vspace{2mm}

        We next compute the relative index class. Using the universal cover $\widetilde{\tor}^n$ gives the isomorphism 
$
K_*(C^*\mathbb{Z}^n)\cong K^*(\tor^n).$
Moreover, under the Chern character map, we have that 
$
\operatorname{Ind}(D_{\widetilde{\tor}^n})
=[\tor^n]\in H^n(\tor^n;\mathbb{Q}).
$

It remains to compute the index class
$\operatorname{Ind}(D_{\widehat{\tor}^n})$. Recall that
\(\widehat{\tor}^n\cong I\times  \tor^{n,\Gamma}\) and  \(\mathbb{Z}^n/\Gamma\) acts freely and
transitively on these components. It follows that
		\begin{equation}
			\ind(D_{\widehat \tor^n})=i_*(\ind(D_{\tor^{n,\Gamma}})),
		\end{equation}
		where $\ind(D_{\tor^{n,\Gamma}})$ is the $\Gamma$-equivariant index with values in $K_*(C^*\Gamma)$, and $i_*$ is induced by the inclusion map $i\colon C^*\Gamma\to C^*\Z^n$. Moreover, since  $C^*\Z^n\cong C(\tor^n)$ and $C^*\Gamma\cong C(\hat \Gamma)$, where $\hat\Gamma$ is the dual of $\Gamma$,  the
homomorphism \(i\colon C^*\Gamma\to C^*\mathbb{Z}^n\) is induced by the
dual map $i^*\colon \tor^n\to\hat \Gamma$, which implies that 
		\begin{equation}
			i_*(\ind(D_{\tor^{n,\Gamma}}))=\begin{cases}
				\textup{a degree } r \textup{ cohomology class},&r<n\\
				|\Z^n/\Gamma|\cdot[\tor^n],&r=n.
			\end{cases}
		\end{equation}
        where $r$ is the rank of $\Gamma$. Consequently, we conclude that from \eqref{relative-index}
		\begin{equation}
			\ind(B)=\ind(D_{\widetilde \tor^n})-\ind(D_{\widehat \tor^n})\ne 0\textup{ in }K_*(C^*\Z^n).
		\end{equation}
	\end{proof}

    \subsubsection{Proof of Theorem \ref{thm:torus}}  Suppose the contrary that there is a complete metric $g$ on $\tor^n\setminus \mathcal{S}$ with $R_{\mu, f}(g)>0$, where $\mu>\frac{n-1}{n}$. There is a positive constant $\lambda$ such that 
    \begin{equation}
        \mu=\frac{n+\lambda-1}{n+\lambda}
    \end{equation}

    Under the notation of Sections \ref{subsubsec:covering} and \ref{subsubsec:index}, we assume that 
$\chi_{1-\bm\rho}$ and $\chi_{\bm\rho}$ are the characteristic functions of the supports of $(1-\bm\rho)$ and $\bm\rho$, respectively. For simplicity of notation, we will abuse the notations and also write $f$ and $g$ for their  lifts to $\widetilde{M}\sqcup\widehat{M}$. We introduce a new connection 
	\begin{equation}\label{new-connection}
		\nabla^f_X=\nabla_X+\frac 1 2\langle X,\nabla f\rangle.
	\end{equation}
We use the this connect to get the Lichnerowicz formula for the operator $B$.
	\begin{lemma}\label{lemma:Lichnerowicz}
        For any $f\in C^\infty(M)$ and $\sigma\in H_0$, 
		\begin{align*}
			\|B\sigma\|^2\geq \frac{\lambda}{4(1+\lambda)}\int_{ M} R_{\mu,f}(g)\cdot|\sigma|^2+\varepsilon^2\|\chi_{\rho}\sigma\|^2-\varepsilon L_\rho\|\chi_{1-\rho}\sigma\|^2
		\end{align*}
	\end{lemma}
	\begin{proof}
		Since $\mathcal U$ anti-commutes with Clifford multiplication, we gave  that 
		\begin{equation}\label{eq:B^2}
			\|B\sigma\|^2\geq \|D\sigma\|^2+\varepsilon^2\|\chi_{\rho}\sigma\|^2-\varepsilon L_\rho\|\chi_{d\rho}\sigma\|^2,
		\end{equation}
		where $\sigma \in H_0$ and $L_\rho$ is defined in \eqref{cut-off-lip-constant}. The  Lichnerowicz formula yields 
		\begin{equation}\label{eq:D}
			\|D\sigma\|^2\geq \|\nabla\sigma\|^2+\frac 1 4\int_{\widetilde M\sqcup\widehat{M}}R(g)|\sigma|^2.
		\end{equation}
        On the other hand, using the connection \eqref{new-connection} and
integrating by parts obtain
		\begin{equation*}
			\begin{split}
				\|\nabla^f\sigma\|^2=&\sum_{i=1}^n\int_{\widetilde M\sqcup\widehat{M}}|\nabla_{e_i}\sigma+\frac 1 2\langle e_i,\nabla f\rangle\sigma|^2\\
				=&\|\nabla\sigma\|^2+\frac 1 4\int_{\widetilde M\sqcup\widehat{M}}|\nabla f|^2|\sigma|^2+\frac 1 2\int_{\widetilde M\sqcup\widehat{M}}\big(\langle\nabla_{e_i}\sigma,\langle e_i,\nabla f\rangle\sigma\rangle+\langle\langle e_i,\nabla f\rangle\sigma,\nabla_{e_i}\sigma\rangle\big)\\
                =&\|\nabla\sigma\|^2+\frac 1 4\int_{\widetilde M\sqcup\widehat{M}}|\nabla f|^2|\sigma|^2+\frac 1 2\int_{\widetilde M\sqcup\widehat{M}}
                \langle\nabla f,\nabla(|\sigma|^2)\rangle\\
                =&\|\nabla\sigma\|^2+\frac 1 4\int_{\widetilde M\sqcup\widehat{M}}\Big(|\nabla f|^2-2\Delta f\Big)|\sigma|^2.
			\end{split}
		\end{equation*}
Combining it with \eqref{eq:D} yields  
		\begin{equation}\label{eq:DNablaf}
			\begin{split}
				\|D\sigma\|^2\geq& \|\nabla^f\sigma\|^2+\frac 1 4\int_{\widetilde M\sqcup\widehat{M}}\Big(R(g)+2\Delta f-|\nabla f|^2\Big)|\sigma|^2.
			\end{split}
		\end{equation}

        \vspace{2mm}

        Moreover, we have  $\sum_{i=1}^n c(e_i)\nabla^f_{e_i}=D+\frac 1 2c(\nabla f)$. Consequently, the Cauchy--Schwarz inequality gives 
		\begin{equation*}
			(1+\frac{n}{\lambda})\left(|D\sigma|^2+ \lambda \big|\nabla^f\sigma\big|^2\right)= (1+\frac{n}{\lambda}) \left(|D\sigma|^2+\lambda\sum_{i=1}^n \big|-c(e_i)\nabla^f_{e_i}\sigma\big|^2\right)\geq\frac 1 4|\nabla f|^2|\sigma|^2. 
		\end{equation*}
        Equivalently, 
       \begin{equation}\label{eq:CauchySchwarz}
           \big|\nabla^f\sigma\big|^2 \geq \frac{1}{4(n+\lambda)} |\nabla f|^2|\sigma|^2  - \frac{1}{\lambda}|D\sigma|^2
        \end{equation}
		Combining \eqref{eq:DNablaf} with \eqref{eq:CauchySchwarz}, we obtain
		\begin{equation}\label{eq:DNablaf2}
			\begin{split}
				\left(1+\frac 1 \lambda\right)\|D\sigma\|^2\geq &\frac 1 4\int_{\widetilde M\sqcup\widehat{M}}\Big(R(g)+2\Delta f-\frac{n+\lambda-1}{n+\lambda}|\nabla f|^2\Big)|\sigma|^2\\
                = &\frac{1}{4}\int_{\widetilde M\sqcup\widehat{M}}R_{\mu,f}(g)|\sigma|^2.
			\end{split}
		\end{equation}
    Substituting \eqref{eq:DNablaf2} into \eqref{eq:B^2} gives the desired
estimate.
	\end{proof}

   We now complete the proof of Theorem \ref{thm:torus}. Since
$1-\bm{\rho}$ has compact support, there exists a constant
$\delta>0$ such that

$$
R_{\mu,f}(g)>\delta
$$

on $\operatorname{supp}(1-\bm{\rho})$. Moreover,
$
R_{\mu,f}(g)>0
$
on $\widetilde M\sqcup\widehat M$.

By Lemma \ref{lemma:Lichnerowicz}, if $\varepsilon>0$ is chosen so that $
\varepsilon L_\rho<\frac{\delta}{2},
$
then the operator $B$ is invertible. This contradicts
Lemma \ref{lemma:index}, which asserts that its index is nonzero. The
contradiction proves Theorem \ref{thm:torus}.

	\begin{remark}
		Following the arguments of \cite{Wang25,WangZhu24}, one obtains the following rigidity statement: if the \((\mu,f)\)-weighted scalar curvature of \((M,g)\) is nonnegative, then \(g\) is Ricci-flat and \(f\) is constant.
	\end{remark}

To conclude this subsection, we provide an example demonstrating that the topological assumption $$i_*\colon\pi_1(S)\to\pi_1(\tor^n)\textup{ is not surjective }$$in Theorem \ref{thm:torus} is  necessary.

\begin{proposition}\label{prop:largePi1}
        Let $\tor^n$ be the $n$ torus, and $\Sigma_k$ the standard $k$-skeleton of $\tor^n$. Then $\tor^n\setminus \Sigma_k$ admit a complete metric with uniform positive scalar curvature if $n\geq 3$ and $k\geq 2$.
    \end{proposition}
    \begin{proof}
        Let $\tor^n=(\sph^1)^n$, and we fix a base point $*\in\sph^1$. Then the $k$-skeleton of $\tor^n$ is given by
		$$\Sigma_k=\bigcup_{\binom{n}{k}}(\sph^1)^k\times\{*\}^{n-k}.$$
		Let $N(\Sigma_k)$ be a small open neighborhood of $\Sigma_k$, and define $W_{k}\coloneqq \tor^n\setminus N(\Sigma_k)$ as the resulting manifold with boundary. We claim that $W_k$ admits a positive scalar curvature metric with a product structure near $\partial W_k$, provided $n\geq 3$ and $k\geq 2$. We construct this metric by induction. 
		
		For the base case $k=n-1$, $W_{n-1}$ is diffeomorphic to the $n$-ball $B^n$. It admits a positive scalar curvature metric with a product structure near the boundary because $\partial B^n=\sph^{n-1}$ and $n-1\geq 2$.
		
		Inductively, $W_{k}$ is obtained from $W_{k+1}$ by performing $\binom{n}{k+1}$ surgeries. Each surgery attaches a $B^{k+1}\times B^{n-k-1}$ to a $B^{k+1}\times \sph^{n-k-2}\subset \partial W_{k+1}$. On the boundary, this process removes a $B^{k+1}\times \sph^{n-k-2}$ and attaches an $\sph^{k}\times B^{n-k-1}$. More precisely, as $(n-1)$-dimensional closed manifolds, $\partial W_{k}$ is obtained from $\partial W_{k+1}$ by performing $\binom{n}{k+1}$ surgeries, each step of which removes a $B^{k+1}\times \sph^{n-k-2}$ and attaches an $\sph^{k}\times B^{n-k-1}$. Meanwhile, the interior topology evolves as $W_k=W_{k+1}\cup_{\partial W_{k+1}}T_k$, where $T_k$ is the trace of the surgeries. 
        
        By the Gromov--Lawson surgery construction, if $k\geq 2$, then $\partial W_k$ admits a uniformly positive scalar curvature metric constructed from the positive scalar curvature metric on $\partial W_{k+1}$ by the induction hypothesis. Moreover, by \cite{Gajer87}, the positive scalar curvature metric on $\partial W_{k+1}$ by the induction hypothesis extends to a positive scalar curvature metric on $T_k$, which is product near boundary. 
        Because the metric on $W_{k+1}$ is also a product near its boundary by the inductive hypothesis, we can cleanly glue these components together. Therefore, $W_k$ admits a positive scalar curvature metric that is a product near the boundary, completing the induction and the proof.
    \end{proof}

\subsection{Llarull-type theorem for weighted scalar curvature}
In this subsection, we prove Theorem \ref{thm:Llarull} and the weighted scalar curvature extremality for enlargeable manifolds.

\subsubsection{Proof of Theorem \ref{thm:Llarull}} The proof of Theorem \ref{thm:Llarull} combines the weighted scalar curvature with a weighted Lichnerowicz estimate for Callias operators. Analagous arguments also appeared in
\cite{Listing10,Llarull98,Deng21,ZhouZhu25}.

Suppose to the contrary  that $M^n$ carries a complete metric $g$ and $f\in C^\infty(M)$ satisfying 
\begin{equation}
R_{\mu, f}(g)-2|\wedge^2 d\Phi|_{\tr}>0
\end{equation}
 where some $\mu>\frac{n-1}{n}$ and $\Phi: M^n\longrightarrow \mathbb{S}^n$ is a smooth map and locally constant outside some compact set $K$ with $\deg(\Phi)\neq 0$. There is a positive constant $\lambda$ such that 
 \begin{equation}
     \mu=\frac{n+\lambda-1}{n+\lambda}. 
 \end{equation}

Let us first consider the case where the dimension $n$ of $M^n$ is even. The odd dimensional case is similar. Consider the $\mathbb Z_2$-graded spinor bundle
$
E:=S\bigl(TM\oplus \Phi^*T\sph^n\bigr),
$
and let $D$ be the associated  Dirac operator.  The Lichnerowicz
formula gives
\begin{equation}\label{eq:twisted-Lichnerowicz}
D^2=\nabla^*\nabla+\frac14R(g)+\mathcal R,
\end{equation}
where
$$
\mathcal R
=
\frac12\sum_{i<j}
c(e_i)c(e_j)c'\bigl(d\Phi(e_i)\wedge d\Phi(e_j)\bigr).
$$
Here $\{e_i\}_{i=1}^n$ is a local orthonormal frame for $TM$, and
$c$ and $c'$ denote the Clifford actions associated with $TM$ and
$\Phi^*T\sph^n$, respectively.

\begin{remark} The computation in \cite{Listing10} gives a pointwise upper bound
\begin{equation}\label{curvatuer-ptws-estimatec}
\mathcal R\leq \frac{1}{2}|\wedge^2 d\Phi|_{\tr}.
\end{equation}
Indeed, fix $x\in M$ and let
$\{\lambda_k\}_{k=1}^{n(n-1)/2}$ be the singular values of
$\wedge^2d\Phi_x$. Choose orthonormal bases ${\alpha_k}$ and ${\beta_k}$  with $
\wedge^2d\Phi_x(\alpha_k)=\lambda_k\beta_k$. Hence,

$$
\mathcal R
=
\frac12\sum_{k=1}^{n(n-1)/2}
\lambda_k c(\alpha_k)c'(\beta_k),
\quad 
|\mathcal R|
\leq
\frac12
\big(\sum_{k=1}^{n(n-1)/2}\lambda_k^2\big)^{1/2}
=
\frac12\left|\wedge^2d\Phi_x\right|_{\tr}.
$$

\end{remark}

For any smooth section $\sigma$ of $E$ with compact support, using \eqref{eq:twisted-Lichnerowicz} and  integration by part obtains that 
\begin{equation}
    ||\sigma||^2\geq \int_M ||\nabla\sigma||^2+\frac{R(g)}{4}|\sigma|^2-|\mathcal{R}||\sigma|^2
\end{equation}
Using the same computation for $|\nabla\sigma|^2$ and $|\nabla^f\sigma|^2$ in Lemma \ref{lemma:Lichnerowicz} and combining with \eqref{curvatuer-ptws-estimatec} imply that 
\begin{equation}\label{eq:weighted-Llarull-estimate}
\begin{split}
        \frac{1+\lambda}{\lambda}\|D\sigma\|^2\geq&~\frac 1 4\int_M\left(R(g)-2\Delta f-\frac{n+\lambda-1}{n+\lambda}|\nabla f|^2-2|\wedge^2 d\Phi|_{\tr}\right)|\sigma|^2\\
        =&~\frac 1 4\int_M\left(R_{\mu,f}(g)- 2|\wedge^2 d\Phi|_{\tr}\right)|\sigma|^2.
    \end{split}  
    \end{equation}

We now deform $D$ by a Callias-type potential. Write
\[
M\setminus K=Z_1\sqcup\cdots\sqcup Z_m
\]
as a disjoint union of connected components. Since $\Phi$ is locally constant outside $K$, there are points $p_i\in \mathbb{S}^n$ with $\Phi(Z_i)=p_i$. Over $Z_i$, we consequently have that 
\[
E|_{Z_i}
\cong
S(TM)\,\widehat\otimes\,S(T_{p_i}\sph^n).
\]Choose a unit vector $V_i\in T_{p_i}\sph^n$, and set
$
\mathcal V_i:=\sqrt{-1}\,c'(V_i).
$
Then $\mathcal V_i$ is an odd, self-adjoint, parallel involution on
$E|_{Z_i}$ which anticommutes with Clifford multiplication by vectors in
$TM$.
Choose smooth functions $\rho_i$ supported in $Z_i$ such that
\[
\rho_i\equiv1
\quad\text{on }Z_i\setminus \{x~|~d(x, K)\leq 1\},
\qquad
|\nabla\rho_i|\leq2.
\]
For $\varepsilon>0$, define
\begin{equation}\label{eq:Llarull-Callias}
B:=D+\varepsilon\sum_{i=1}^m\rho_i\mathcal V_i.
\end{equation}
For every $\sigma\in C_c^\infty(M,E)$, the anticommutation relations give
\begin{align}
\|B\sigma\|^2
&=
\|D\sigma\|^2
+\sum_{i=1}^m
\int_M
\left(
\varepsilon
\left\langle [D,\rho_i]\mathcal V_i\sigma,\sigma\right\rangle
+\varepsilon^2\rho_i^2|\sigma|^2
\right) \notag\\
&\geq
\|D\sigma\|^2
-2\varepsilon\sum_{i=1}^m
\int_{\{\nabla\rho_i\neq0\}}|\sigma|^2
+\varepsilon^2\sum_{i=1}^m
\int_{\{\rho_i=1\}}|\sigma|^2.
\label{eq:Callias-lower-bound}
\end{align}
Together with \eqref{eq:weighted-Llarull-estimate}, this yields
\begin{align}
\|B\sigma\|^2
\geq{}&
\frac{\lambda}{4(1+\lambda)}
\int_M
\left(
R_{\mu,f}(g)
-2\left|\wedge^2d\Phi\right|_{\tr}
\right)|\sigma|^2 \notag\\
&\quad
-2\varepsilon\sum_{i=1}^m
\int_{\{\nabla\rho_i\neq0\}}|\sigma|^2
+\varepsilon^2\sum_{i=1}^m
\int_{\{\rho_i=1\}}|\sigma|^2.
\label{eq:Callias-final-bound}
\end{align}

The potential is uniformly invertible outside $\{x~|~d(x, K)\leq 1 \}$; hence $B$ is
Fredholm. Moreover, the relative index theorem gives
\[
\operatorname{ind}(B)
=
\deg(F)\,\chi(\sph^n)\neq0.
\]
On the other hand, the sets $\{\nabla\rho_i\neq0\}$ are compact. The first term on the right-hand side of
\eqref{eq:Callias-final-bound} has a positive lower bound on their union.
Thus, for $\varepsilon>0$ sufficiently small, there is a constant
$c_\varepsilon>0$ such that
\[
\|B\sigma\|^2\geq c_\varepsilon\|\sigma\|^2
\]
for every $\sigma\in C_c^\infty(M,E)$. It follows that $B$ is
invertible, contradicting $\operatorname{ind}(B)\neq0$. This proves the
theorem when $n$ is even.

When $n$ is odd, the same construction, together with the spectral-flow
argument of \cite{LiSuWang24}, yields the desired contradiction.

\subsubsection{Area-enlargeable manifolds}

As an application of Theorem \ref{thm:Llarull}, we obtain an extremality
result for area-enlargeable manifolds, extending the corresponding result
of \cite{Su26}. We begin by recalling the relevant definition, allowing
the underlying manifold to be noncompact.

\begin{definition}\label{def:enlargeable}
Let $M^n$ be a connected manifold without boundary. A Riemannian metric
$g$ on $M$ is called $\wedge^2$-enlargeable if, for every
$\varepsilon>0$, there exist a spin covering
$
\pi_\varepsilon\colon M_\varepsilon\longrightarrow M
$
and a smooth map
$
\Phi_\varepsilon\colon
(M_\varepsilon,\pi_\varepsilon^*g)\longrightarrow(\sph^n,g_{\sph^n})
$
with $\deg\Phi_\varepsilon\neq 0$ , which  is locally constant outside a
compact set, and satisfies
\[
\left|\wedge^2d\Phi_\varepsilon\right|\leq\varepsilon.
\]
The manifold $M$ is called $\wedge^2$-enlargeable if every
Riemannian metric on $M$, not necessarily complete, is
$\wedge^2$-enlargeable.
\end{definition}

\begin{corollary}\label{coro:enlargeable}
Let $M^n$ be a $\wedge^2$-enlargeable manifold without boundary.
Then $M$ admits no complete Riemannian metric $g$ for which
\[
R_{\mu,f}(g)>0
\]
for some $f\in C^\infty(M)$ and
$\mu>\frac{n-1}{n}$.
\end{corollary}

\begin{proof}
Suppose, to the contrary, that $g$ is a complete metric on $M$ and
that
$
R_{\mu,f}(g)>0
$
for some $f\in C^\infty(M)$ and $\mu>\frac{n-1}{n}$. Consider the  metric
\[
g':=R_{\mu,f}(g)\,g.
\]
Although $g'$ need not be complete, it is a Riemannian metric on $M$.
By $\wedge^2$-enlargeability, for every $\varepsilon>0$, there exist a
spin covering
$
\pi_\varepsilon\colon M_\varepsilon\longrightarrow M
$
and a smooth map
$
\Phi_\varepsilon\colon
(M_\varepsilon,\tilde{g})
\longrightarrow
(\sph^n,g_{\sph^n})
$ with nonzero degree, is locally constant outside a compact set, and
satisfies
\[
\left|\wedge^2d\Phi_\varepsilon\right|_{\pi_\varepsilon^* g'}
\leq\varepsilon \qquad \widetilde g'
=
R_{\mu,\widetilde f}(\widetilde g)\,\widetilde g\]
where 
$
\widetilde g:=\pi_\delta^*g$ and $
\widetilde f:=\pi_\delta^*f
$. 
Regard $\Phi_\varepsilon$ as a map defined on
$(M_\varepsilon,\widetilde g)$, and denote the resulting map by
$\Phi_\varepsilon$. Choose
$0<\varepsilon<\frac{1}{n(n-1)}$ and obtain
\[
\left|\wedge^2d\Phi_\varepsilon\right|_{\widetilde g}
\leq
\varepsilon\,R_{\mu,\widetilde f}(\widetilde g),
\qquad
R_{\mu,\widetilde f}(\widetilde g)
-n(n-1)|\wedge^2d\Phi_\varepsilon|_{\tilde g}
\geq
R_{\mu,\widetilde f}(\widetilde g)
\bigl(1-n(n-1)\delta\bigr)
>0.\]
Since
$
2|\wedge^2d\Phi_\varepsilon|_{\tr}
\leq
n(n-1)\left|\wedge^2d\Phi_\varepsilon\right|,
$
,we obtain
$
R_{\mu,\widetilde f}(\widetilde g)
-
2|\wedge^2d\Phi_\varepsilon|_{\tr}
>0.
$
The lifted metric $\widetilde g$ is complete, whereas $\Phi_\varepsilon$ has
nonzero degree and is locally constant outside a compact set. This
contradicts Theorem \ref{thm:Llarull}.
\end{proof}

We conclude this section by comparing the preceding noncompact results with the closed
case.

\begin{remark}\label{rem:closed-conformal}
On a closed manifold, weighted scalar-curvature extremality follows
directly from the classical scalar-curvature results under the weaker
condition
$
\mu>\frac{n-2}{n-1}.
$

Indeed, for \(f\in C^\infty(M)\), the conformal metric
$
g_f:=e^{-\frac{2f}{n-1}}g
$
satisfies
\begin{align}
R(g_f)
&=
e^{\frac{2f}{n-1}}
\left(
R(g)+2\Delta f-\frac{n-2}{n-1}|\nabla f|^2
\right) \notag\\
&=
e^{\frac{2f}{n-1}}
\left(
R_{\mu,f}(g)
+
\left(\mu-\frac{n-2}{n-1}\right)|\nabla f|^2
\right).
\label{eq:weighted-conformal-change}
\end{align}
Thus, if \(R_{\mu,f}(g)>0\) and
\(\mu>\frac{n-2}{n-1}\), then \(g_f\) has positive scalar curvature.

Similarly, suppose that
$
\Phi\colon(M,g)\longrightarrow(\sph^n,g_{\sph^n})
$
satisfies
$$
R_{\mu,f}(g)
-
2\left|\wedge^2d\Phi\right|_{\tr}
>0.
$$
Viewing the same map as $
\Phi\colon(M,g_f)\longrightarrow(\sph^n,g_{\sph^n}),$
we have
$$
\left|\wedge^2dF\right|^{g_f}_\tr
=
e^{\frac{2f}{n-1}}
\left|\wedge^2dF\right|^g_\tr.
$$
Consequently,
\begin{align*}
R(g_f)
-
2\left|\wedge^2d\Phi\right|^{g_f}_{\tr}
=
e^{\frac{2f}{n-1}}
\Bigg(
R_{\mu,f}(g)
-
2\left|\wedge^2d\Phi\right|^g_\tr+
\left(\mu-\frac{n-2}{n-1}\right)|\nabla f|^2
\Bigg)>0.
\end{align*}
The classical scalar-curvature extremality theorems therefore apply
directly.

This conformal reduction argument does not extend immediately to the noncompact
setting, since the metric \(g_f\) need not be complete even when \(g\) is
complete. The direct arguments in Theorems \ref{thm:torus} and
\ref{thm:Llarull} instead require the stronger threshold
$
\mu>\frac{n-1}{n}.
$ This threshold is in fact optimal in the noncompact setting, which will be discussed elsewhere.
\end{remark}
	
\section{Blow-up of Singular set}
In this section, we use conformal deformations to construct complete
metrics of positive weighted scalar curvature from $L^\infty$-metrics
whose scalar curvature is positive on the regular set. We begin with a
general blow-up result for the singular set.

\begin{proposition}\label{prop:weightedSc}
Let $(M^n, g)$ be a closed $L^\infty$-manifold and $\mathcal{S}$ be the singular set satisfying the Hausdorff dimension upper bound 
\begin{equation}
    \dim_{\mathcal{H}}(\mathcal{S})\leq n-3+\frac{2}{n}. 
\end{equation}
If the scalar curvature $R(g)$ is positive on $M\setminus \mathcal{S}$, then there exist a constant $\mu>\frac{n-1}{n}$, a function $f\in C^\infty(M\setminus \mathcal{S})$ and a complete metric $g'$ on $M\setminus \mathcal{S}$ satisfying that on $M^n\setminus \mathcal{S}$
\begin{equation}
    R_{\mu, f}(g')>0
\end{equation}
\end{proposition}

When $\mathcal S$ is an embedded simplicial subcomplex of dimension at
most $n-3$, the preceding construction can be refined to produce a
complete conformal metric with quantitative scalar-curvature estimates
near the singular set.

\begin{proposition}\label{simplical-blow-up}
    Let $(M^n, g)$ be a closed $L^\infty$-manifold and $\mathcal{S}$ be the singular set which is an embedded simplical subcomplex of dimension at most $n-3$. If the scalar curvature $R(g)$ is positive on $M^n\setminus \mathcal{S}$,  then there is  a constant $\mu>\frac{n-1}{n}$ with the followint property. For any $\varepsilon>0$, there exist constants $a>0$ and $C>0$,  functions $u_\varepsilon, f_\varepsilon\in C^\infty(M\setminus \mathcal{S})$ and a neighborhood $U_1$ of $\mathcal S$ such that 
    \begin{itemize}
    \item[(1)] $u_\varepsilon>0$, $u_\varepsilon=1$ outside a neighborhood of $\mathcal{S}$ and \[u_\varepsilon(x)\rightarrow +\infty \text{ as } x\rightarrow \mathcal{S};\]
    \item[(2)]  the conformal  metric $g'_\varepsilon=u^{2a}_\varepsilon g$  is complete on $M\setminus \mathcal{S}$;
    \item[(3)] the weighted scalar curvature of $g'_\varepsilon$ satisfies 
    \begin{equation}\label{weighted-sc-curvature-ineq}
        R_{\mu, f_\varepsilon}(g'_\varepsilon)\geq u_\varepsilon^{-2a}(R(g)-\varepsilon) \text{ on } M\setminus U_1  \textup{ and } R_{\mu, f_\varepsilon}(g'_\varepsilon)\geq C u_\varepsilon^{-2a}d(x,\mathcal{S})^{-2} \text{ on } U_1,
    \end{equation} 
    \end{itemize}
\end{proposition}

\subsection{Laplace-Beltrami operators for  {$L^\infty$}-metrics} 
Let $(M^n,g)$ be a closed
$L^\infty$-manifold with background metric $g_0$, and let $\mathcal S$ denote its singular set.  In order to build
the desired conformal factor, we will employ Green's functions with pole at points of $\mathcal S$. Choose precompact open neighborhoods $U_1$ and $U_2$ of $\mathcal{S}$, such that $\partial U_2$ is smooth and 
\begin{equation}
        \mathcal S\subset U_1,\qquad
        \overline{U}_1\subset U_2 .
\end{equation}
The existence of Green's functions for uniformly elliptic operators with $L^\infty$ coefficients, and their basic properties, has been obtained in \cite[Theorem~1.1]{MR657523}, see also \cite{MR161019} and \cite[Lemma~A.1]{ChengLeeTam2022}. We recall the relevant form of this result, needed for our purposes, here. 

\begin{lemma}[Existence and estimates for the Green's function]
\label{exist-Green}
Let \((U_2,g|_{U_2})\) and \(\mathcal S\) be as above.  Then there exists a
Dirichlet Green's function
\begin{equation}
        G:U_2\times U_2\to \mathbb R\cup\{\infty\}
\end{equation}
associated with $-\Delta_g$ on \(U_2\).  Moreover, for each
\(x\in \mathcal S\), the function
\begin{equation}
        y\mapsto G(x,y)
\end{equation}
is positive and \(g\)-harmonic on \(U_2\setminus \mathcal S\).  In addition,
there exists a constant \(C>0\), independent of \(x\in\mathcal S\), such that
for all \(x\in\mathcal S\) and \(y\in U_1\setminus\mathcal S\),
\begin{equation}\label{f0a-9fjh-9qjhh}
        C^{-1} d_{g_0}(x,y)^{2-n}
        \leq G(x,y)
        \leq C d_{g_0}(x,y)^{2-n}.
\end{equation}
On the regular set \(U_2\setminus\mathcal S\), the function \(G(x,\cdot)\) is
smooth.
\end{lemma}
 \subsection{Blow-up of the singular set with low regularity}
 In this subsection, we  construct a Green's function with pole $\mathcal{S}$ and use it to perform a conformal change to blow up the metric near the singular set $\mathcal{S}$. This approach was originally  proposed in \cite[Lemma 2.3]{WangXie24} and further developed in \cite{BiHaoHeShiZhu26}.

\begin{proof}[Proof of Proposition \ref{prop:weightedSc}]\ 

\noindent\emph{Step 1: Construction of a Green's function having pole $\mathcal S$.} 
Let $\sigma$ be a positive number in the interval $ (\dim_{\mathcal{H}}\mathcal{S}, n-3+\frac{2}{n})$, and choose a value $a$
satisfying
\begin{equation}
   \max\left\{1,\frac{2}{n-2},\frac{1}{\,n-2-\sigma\,}\right\} <a< \frac{n}{n-2}.
\end{equation}
%Let $a\in (1,\frac{n}{n-2})$, and choose $\sigma\in (n-3,n-2-a^{-1})$.
Since $\dim_{\mathcal H}\mathcal{S}< \sigma$ and $\mathcal{S}$ is compact, there exists
$\ell_0$ such that, for every integer $\ell\ge \ell_0$, one can find a
finite cover $\{B^{g_0}_{r_{\ell,j}}(x_{\ell,j})\}_{j=1}^{N_\ell}$ of $\mathcal S$ by $g_0$-geodesic balls,
with $x_{\ell,j}\in \mathcal S$, satisfying
\begin{equation}\label{cover-choice}
        r_{\ell,j}\le 2^{-\ell},
        \qquad\quad
        \sum_{j=1}^{N_\ell} r_{\ell,j}^{\sigma}\le 1 .
\end{equation}
After reindexing $\ell$, we may assume this holds for all integers $\ell\ge 1$.
Define the atomic measure
\begin{equation}
\nu_{\mathcal S}\coloneqq
\sum_{\ell=1}^{\infty}\sum_{j=1}^{N_\ell}
r_{\ell,j}^{\,n-2-a^{-1}}\delta_{x_{\ell,j}} .
\end{equation}
Set $p:=n-2-a^{-1}$, and observe that since $p-\sigma>0$ it follows from \eqref{cover-choice} that for any measurable set $X\subset U_2$ we have
\begin{align}
\begin{split}
\nu_S(X)
\leq
\sum_{\ell=1}^{\infty}\sum_{j=1}^{N_\ell}
r_{\ell,j}^{p} &=
\sum_{\ell=1}^{\infty}\sum_{j=1}^{N_\ell}
r_{\ell,j}^{p-\sigma}r_{\ell,j}^{\sigma}  \\
&\leq
\sum_{\ell=1}^{\infty}
2^{-\ell(p-\sigma)}
\sum_{j=1}^{N_\ell}r_{\ell,j}^{\sigma}  
\leq
\sum_{\ell=1}^{\infty}2^{-\ell(p-\sigma)}
<\infty .
\end{split}
\end{align}
Thus $\nu_{\mathcal S}$ is a finite measure. Moreover, since each atom is centered at a point
$x_{\ell,j}\in \mathcal S$ and $\mathcal S$ is closed, $\nu_{\mathcal S}$ is supported on $\mathcal S$.

Next set
\begin{equation}
        G_{\mathcal{S}}(y):=\int_{\mathcal{S}} G(x,y)\,d\nu_{\mathcal S}(x)
        =
        \sum_{\ell=1}^{\infty}\sum_{j=1}^{N_\ell}
        r_{\ell,j}^{\,n-2-a^{-1}}G(x_{\ell,j},y).
\end{equation}
Since $\nu_{\mathcal S}$ is finite and the Green's functions are locally uniformly
bounded for $y$ in compact subsets of $U_2\setminus \mathcal{S}$ by Lemma \ref{exist-Green}, this series
converges locally uniformly on $U_2\setminus \mathcal S$. Hence $G_{\mathcal S}$ is a
positive harmonic function on $U_2\setminus \mathcal S$. Furthermore, by interior elliptic
regularity, $G_{\mathcal S}$ is smooth on $U_2\setminus \mathcal S$.

We note that $G_{\mathcal S}$ has pole set $\mathcal S$. Indeed, fix
$x\in \mathcal S$. For each $\ell\geq 1$, choose $j_\ell$ such that
$x\in B^{g_0}_{r_{\ell,j_\ell}}(x_{\ell,j_\ell})$,
and set $s_\ell:=r_{\ell,j_\ell}$. Then $s_\ell\leq 2^{-\ell}$, and hence
$s_\ell\to0$. If $y\in U_2\setminus \mathcal S$ satisfies
$d_{g_0}(y,x)<s_\ell$, then
\begin{equation}
        d_{g_0}(y,x_{\ell,j_\ell})
        \leq d_{g_0}(y,x)+d_{g_0}(x,x_{\ell,j_\ell})
        <2s_\ell .
\end{equation}
Using the lower Green's function estimate \eqref{f0a-9fjh-9qjhh} and the definition of $\nu_{\mathcal S}$, we obtain
\begin{align}
\begin{split}
        G_{\mathcal S}(y)
        \geq
        s_\ell^{\,n-2-a^{-1}}G(x_{\ell,j_\ell},y)  
        &\geq
        C^{-1}s_\ell^{\,n-2-a^{-1}}
        d_{g_0}(x_{\ell,j_\ell},y)^{2-n}  \\
        &\geq
        C^{-1}2^{2-n}
        s_\ell^{\,n-2-a^{-1}}s_\ell^{\,2-n}  
        =
        c\,s_\ell^{-a^{-1}} .
\end{split}
\end{align}
It follows that
\begin{equation}
        G_{\mathcal S}(y)\to+\infty
        \qquad\text{as } y\to x,\quad y\in U_2\setminus\mathcal S .
\end{equation}
For the completeness of the conformally blown-up metric, however, this
pointwise blow-up is not sufficient by itself. What is needed is the stronger
integral statement that $G_{\mathcal S}^a$ is non-integrable along every curve
approaching $\mathcal S$. The Wolff-type potential introduced below encodes
the scale-by-scale lower bounds for the Green's potential and will be used to
prove precisely this nonintegrability.

%We next show that $G_{\mathcal S}$ blows up sufficiently rapidly near $\mathcal S$. 
Consider the Wolff-type potential on $U_2$,
\begin{equation}
        W(x):=\int_0^1\left(\nu_{\mathcal{S}}(B^{g_0}_r(x))r^{2-n}\right)^a\,dr .
\end{equation}
Fix $x\in \mathcal S$, and as above for each $\ell\ge1$ choose $j_\ell$ such that
$x\in B^{g_0}_{r_{\ell,j_\ell}}(x_{\ell,j_\ell})$.
%Set $s_\ell:=r_{\ell,j_\ell}$. Then $s_\ell\le 2^{-\ell}$, and hence
%$s_\ell\to0$. 
Note that for every $r\in [s_\ell,2s_\ell]$, one has
$x_{\ell,j_\ell}\in B^{g_0}_r(x)$. Therefore
\begin{equation}
        \nu_{\mathcal S}(B^{g_0}_r(x))\ge s_\ell^{\,n-2-a^{-1}} .
\end{equation}
It follows that, for $r\in [s_\ell,2s_\ell]$,
\begin{equation}
\left(\nu_{\mathcal S}(B^{g_0}_r(x))r^{2-n}\right)^a
\ge
c\,\left(s_\ell^{\,n-2-a^{-1}}s_\ell^{\,2-n}\right)^a
=
c\,s_\ell^{-1}.
\end{equation}
Since $s_\ell\to0$, we may choose a subsequence $\ell_k$ such that $2s_{\ell_{k+1}}<s_{\ell_k}$,
then the intervals $[s_{\ell_k},2s_{\ell_k}]$ are pairwise disjoint. Consequently,
\begin{equation}\label{blow-up-wolff}
W(x)
\ge
\sum_{k=1}^\infty
\int_{s_{\ell_k}}^{2s_{\ell_k}}
\left(\nu_{\mathcal S}(B^{g_0}_r(x))r^{2-n}\right)^a\,dr  
\ge
c\sum_{k=1}^\infty
\int_{s_{\ell_k}}^{2s_{\ell_k}}s_{\ell_k}^{-1}\,dr
= +\infty .
\end{equation}
Thus $W$ %$W(x)=+\infty$ 
takes the value $+\infty$ at all points of $\mathcal S$.

\medskip
\noindent\emph{Step 2: Construction of a complete metric on $M^n\setminus \mathcal S$.} 
Let $\varphi$ be a smooth function on $M^n$ satisfying
\begin{equation}
        0\leq \varphi\leq 1,\qquad
        \varphi=1 \text{ on } U_1,\qquad
        \overline{\operatorname{supp}\varphi}\subset U_2 .
\end{equation}
For a constant $\tau>0$, to be chosen later, define
\begin{equation}
        w:=1+\tau\varphi G_{\mathcal S},
        \qquad
        g':=w^{2a}g ,
\end{equation}
where $\varphi G_{\mathcal S}$ is extended by zero outside $U_2$.

We now prove completeness near $\mathcal S$. Let
$\gamma:[0,L)\to U_1\setminus \mathcal S$ be a $g$-unit-speed curve such
that $\lim_{t\to L}\gamma(t)=x\in \mathcal S$.
Since $g$ and $g_0$ are uniformly equivalent on $\overline{U}_2$, there
is a constant $C_0>0$ such that, for $t\in[L-\eta,L)$ with $\eta>0$ sufficiently small, we have
$d_{g_0}(\gamma(t),x)\leq C_0(L-t)$.
Set $r_t:=2C_0(L-t)$. If $z\in B_{r_t}^{g_0}(x)$, then
\begin{equation}
        d_{g_0}(z,\gamma(t))
        \leq d_{g_0}(z,x)+d_{g_0}(x,\gamma(t))
        \leq 2r_t .
\end{equation}
Using the lower Green's function estimate \eqref{f0a-9fjh-9qjhh}, and the fact that $\nu_{\mathcal S}$ is supported on $\mathcal{S}$, we obtain
\begin{align}
\begin{split}
        G_{\mathcal S}(\gamma(t))
        &=\int_{\mathcal S}G(z,\gamma(t))\,d\nu_{\mathcal S}(z) \\
        &\geq
        \int_{B_{r_t}^{g_0}(x)}\!\!
        G(z,\gamma(t))\,d\nu_{\mathcal S}(z) 
        \geq
        c\,\nu_{\mathcal S}(B_{r_t}^{g_0}(x))\, r_t^{2-n}.
\end{split}
\end{align}
Therefore, as in \cite[Lemma 2.20]{BiHaoHeShiZhu26},
\begin{align}
\begin{split}
        \int_0^L G_{\mathcal S}(\gamma(t))^a\,dt
        &\geq
        c\int^{L}_{L-\eta}
        \left(\nu_{\mathcal S}(B_{r_t}^{g_0}(x))r_t^{2-n}\right)^a\,dt \\
        &=
        c(2C_0)^{-1}\int_0^{2C_0\eta}
        \left(\nu_{\mathcal S}(B_r^{g_0}(x))r^{2-n}\right)^a\,dr
        =
        +\infty,
\end{split}
\end{align}
where the last equality follows from $W(x)=+\infty$.

Since $\varphi=1$ on $U_1$, we have $w\geq \tau G_{\mathcal S}$ along
$\gamma$ for $t$ sufficiently close to $L$. Hence
\[
        L_{g'}(\gamma)
        =
        \int_0^L w(\gamma(t))^a\,dt
        \geq
        \tau^a\int_{L-\eta}^L G_{\mathcal S}(\gamma(t))^a\,dt
        =
        +\infty .
\]
Thus every curve approaching $\mathcal S$ has infinite $g'$-length.
It follows that $g'$ is complete on
$M^n\setminus\mathcal S$.

\medskip
\noindent\emph{Step 3: The positivity of the weighted scalar curvature.} Fix $\varepsilon>0$, to be chosen sufficiently small, and set
\begin{equation}
        \mu:=\frac{n-1}{n}+\varepsilon,
        \qquad
        f:=b\log w,
\end{equation}
where $b\in\mathbb R$ will be chosen below. 
Since $g'=w^{2a}g$, the conformal transformation formulas give
\begin{align}
\begin{split}
        &w^{2a}R(g')
        =
        R(g)-2a(n-1)\Delta_g\log w
        -(n-1)(n-2)a^2|\nabla\log w|_g^2, \\
        &w^{2a}\Delta_{g'}\log w
        =
        \Delta_g\log w
        +a(n-2)|\nabla\log w|_g^2, \\
        &w^{2a}|\nabla\log w|_{g'}^2
        =
        |\nabla\log w|_g^2 .
\end{split}
\end{align}
Using $f=b\log w$, we obtain
\begin{align}
\begin{split}
    w^{2a}R_{\mu, f}(g')
    &=
    R(g)+2\bigl(b-a(n-1)\bigr)\Delta_g\log w          \\
    &\quad
    -
    \left(
        -2ab(n-2)
        +(n-1)(n-2)a^2
        +\mu b^2
    \right)
    |\nabla\log w|_g^2 .
\end{split}
\end{align}
Since
$
        \Delta_g\log w
        =
        \frac{\Delta_g w}{w}
        -
        |\nabla\log w|_g^2
$,
this becomes
\begin{equation}\label{Q-computation-2'}
%\begin{split}
    w^{2a}R_{\mu, f}(g')
    =
    R(g)
    +2\bigl(b-a(n-1)\bigr)\frac{\Delta_g w}{w}       
    +
    %\Bigl[
    %    -\bigl(2a(n-2)-2\bigr)b
    %    -\bigl((n-1)(n-2)a^2-2a(n-1)\bigr)
    %    -\mu b^2
    %\Bigr]
    c_*|\nabla\log w|_g^2 ,
%\end{split}
\end{equation}
where
\begin{equation}\label{b-equation'}
  c_*=  \bigl(2a(n-2)-2\bigr)b
    -\bigl((n-1)(n-2)a^2-2a(n-1)\bigr)
    -\mu b^2.
\end{equation}
We choose $b$ so that the coefficient of the gradient term in
\eqref{Q-computation-2'} vanishes, namely $c_*=0$.
Because $a>\frac{2}{n-2}$
this quadratic equation has a real root provided that
\begin{equation}\label{discriminant-condition'}
    \mu
    <
    \frac{(2a(n-2)-2)^2}
    {4\bigl((n-1)(n-2)a^2-2a(n-1)\bigr)} .
\end{equation}
Moreover, the choice $a<\frac{n}{n-2}$ implies
\begin{equation}
    \frac{(2a(n-2)-2)^2}
    {4\bigl((n-1)(n-2)a^2-2a(n-1)\bigr)}
    >
    \frac{n-1}{n}.
\end{equation}
Hence there exists $\varepsilon_n>0$ such that
\eqref{discriminant-condition'} holds whenever
\begin{equation}
        0<\varepsilon<\varepsilon_n,
        \qquad
        \mu=\frac{n-1}{n}+\varepsilon .
\end{equation}
We remark that 
\begin{enumerate}\label{coff-discc}
    \item For each such $\varepsilon$, choose a real root
$b=b(\varepsilon)$ of \eqref{b-equation'}. The branch may be chosen so that
$b(\varepsilon)$ remains uniformly bounded for
$\varepsilon\in(0,\varepsilon_n)$.
\item When $c_*=0$,  \eqref{b-equation'} implies that 
$b-a(n-1)<0$. 
\end{enumerate}
With this choice of $b$, the gradient term in \eqref{Q-computation-2'}
vanishes, and it follows that
\begin{equation}
    w^{2a}R_{\mu, f}(g')
    =
    R(g)
    +2\bigl(b-a(n-1)\bigr)
    \frac{\tau\Delta_g(\varphi G_{\mathcal S})}{w}.
\end{equation}
Because $w\geq1$, and because $b=b(\varepsilon)$ is uniformly bounded
for $\varepsilon\in(0,\varepsilon_n)$, there exists a constant
$C_1>0$, independent of $\varepsilon$, such that
\begin{equation}\label{Q-lower'}
    w^{2a}R_{\mu,f}(g')
    \geq
    R(g)
    -
    C_1\tau|\Delta_g(\varphi G_{\mathcal S})| .
\end{equation}
Observe that $\Delta_g(\varphi G_{\mathcal S})=0$ on
$U_1\setminus\mathcal S$, since $\varphi=1$ there and
$G_{\mathcal S}$ is $g$-harmonic. It also vanishes on
$M^n\setminus U_2$, since $\varphi G_{\mathcal S}$ is supported in
$U_2$. On the compact set
$\overline{U_2\setminus U_1}$, define
\begin{equation}
    C_2:=
    \sup_{\overline{U_2\setminus U_1}}
    |\Delta_g(\varphi G_{\mathcal S})|<\infty,
    \qquad
    C_3:=
    \inf_{\overline{U_2\setminus U_1}}R(g)>0 .
\end{equation}
Choose $\tau>0$ sufficiently small so that
$C_1C_2\tau<C_3$.
Then \eqref{Q-lower'} gives
\begin{equation}
        w^{2a}R_{\mu, f}(g')>0
        \qquad\text{on } U_2\setminus U_1 .
\end{equation}
On $U_1\setminus\mathcal S$ and on $M^n\setminus U_2$, we have
$\Delta_g(\varphi G_{\mathcal S})=0$, and hence on those sets $w^{2a}R_{\mu, f}(g')=R(g)>0$.
Therefore, $R_{\mu, f}(g')>0$ on $M^n\setminus\mathcal S$.
\end{proof}

\subsection{Blow-up of Lipchitz embedded simplical subcomplexes} Suppose that the singular set $\mathcal S$ is a Lipschitz embedded
simplicial subcomplex. In this case, the Hausdorff measures on the
simplices of $\mathcal S$ provide a natural alternative to the atomic
measure $\nu_{\mathcal S}$. Following the construction in
\cite{WangXie24}, we use these measures to construct a Green-type function
with pole along $\mathcal S$, and hence a complete conformal metric on
$M\setminus\mathcal S$ with controlled weighted scalar curvature.

Throughout this subsection, $(M^n,g)$ is a closed manifold equipped with
an $L^\infty$-metric, smooth away from a singular set $\mathcal S$,
and $\mathcal S$ is a finite Lipschitz embedded simplicial subcomplex of
dimension at most $n-3$.  We state two
estimates that will be used to control the weighted scalar curvature (Lemma \ref{pt-poission-eq} and Lemma \ref{decay-integ-estimate}). Their
proofs are deferred to the Appendix.

\begin{lemma}\label{pt-poission-eq}
Under the hypotheses of Lemma \ref{exist-Green}, let $p\in\mathcal S$
and $\alpha\in(2,n)$. There exists a function
$
\psi_p\in C^\infty(U_2\setminus\mathcal S)
$
satisfying
\begin{equation}\label{poisson}
-\Delta_g\psi_p=d_g(\,\cdot\,,p)^{-\alpha}
\text{ in }U_2\setminus\mathcal S,\qquad 
\psi_p=0 \text{ on }\partial U_2.
\end{equation}
Moreover, there is a constant $C\geq1$, independent of $p$, such that for every $x\in U_1\setminus\mathcal S$
\begin{equation}\label{two-side-estimate}
C^{-1}d_g(x,p)^{2-\alpha}
\leq
\psi_p(x)
\leq
C\,d_g(x,p)^{2-\alpha}. 
\end{equation}
\end{lemma}

We shall also use the following elementary integration estimate.

\begin{lemma}\label{decay-integ-estimate}
Let $\sigma$ be a compact Lipschitz embedded $k$-simplex, and suppose
that $f$ is a positive continuous function satisfying
\[
C_0^{-1}d_g(x,y)^{-\beta}
\leq
f(x,y)
\leq
C_0d_g(x,y)^{-\beta}.
\]
If $\beta>k$, then there exists $C\geq1$ such that
\begin{equation}\label{eq:simplex-integral-estimate}
C^{-1}d_g(x,\sigma)^{k-\beta}
\leq
\int_\sigma f(x,y)\,d\mathcal H^k(y)
\leq
C\,d_g(x,\sigma)^{k-\beta}.
\end{equation}
\end{lemma}

We now construct the required conformal factor $u$. Consider
 the singular set $\mathcal{S}$ is the union of finite simplices $\{\sigma_j\}$ with $\dim (\sigma_j)\leq n-3$. Fix $\eta\in (0,1)$, to be chosen sufficiently small, and   let $\ell_j=\dim(\sigma_j)$ and $\alpha_j=\ell_j+(3-\eta)$. Since $\ell_j\leq n-3$, we have 
 \begin{equation}
     2<\alpha_j<n. 
 \end{equation}For $p\in \sigma_j$, let $\psi_{j, p}$ be the solution of 
 $-\Delta_g\psi_{j,p}=d_g(\,\cdot\,,p)^{-\alpha_j}$ given by Lemma \ref{pt-poission-eq}. Define
\begin{equation}\label{def-blow-up}
\begin{split}
    \psi_{\mathcal{S}}(x)&=\sum_j\int_{\sigma_j}\psi_p(x)d\mathcal{H}^{\dim\sigma_j}(p) \\
    \phi_{\mathcal{S}}(x)&=\sum_j\int_{\sigma_j} d^{-\alpha_j}(x, p)d\mathcal{H}^{\dim\sigma_j}(p)
    \end{split}
\end{equation} 
It follows that
\begin{equation}\label{integ-poission}
-\Delta_g\psi_{\mathcal S}=\phi_{\mathcal S}
\qquad\text{on }U_2\setminus\mathcal S.
\end{equation}
Because
$
\ell_j-\alpha_j=\eta-3<0
$ and $
\ell_j+2-\alpha_j=\eta-1<0
$, 
Lemma \ref{decay-integ-estimate} gives
for $x\in U_1\setminus\mathcal{S}$
\begin{equation}\label{bound-prescribed-solution}
\phi_{\mathcal S}(x)
\asymp d_g(x,\mathcal S)^{\eta-3}
\qquad 
\psi_{\mathcal S}(x)
\asymp
C\,d_g(x,\mathcal S)^{\eta-1}
\end{equation}
Combining \eqref{def-blow-up},
\eqref{integ-poission} and \eqref{bound-prescribed-solution}, we obtain
\begin{equation}\label{Laplace-estimate}
-\psi_{\mathcal S}^{-1}\Delta_g\psi_{\mathcal S}
=
\psi_{\mathcal S}^{-1}\phi_{\mathcal S}
\geq
C^{-1}d_g(x,\mathcal S)^{-2}.
\end{equation}

\begin{proof}[Proof of Proposition \ref{simplical-blow-up}]
Choose $\eta>0$ sufficiently small that
$
\frac{1}{1-\eta}<\frac{n}{n-2},
$
and then fix $a$ satisfying
\begin{equation}\label{eq:choice-a-eta}
\frac{1}{1-\eta}<a<\frac{n}{n-2}.
\end{equation}
Choose $\varepsilon_0>0$ sufficiently small and set
$
\mu:=\frac{n-1}{n}+\varepsilon_0.
$
Let $b=b(\varepsilon_0)$ be the solution of \eqref{b-equation'} chosen
so that the gradient term in the conformal transformation formula
vanishes. From Remark \ref{coff-discc}, we have that 
\[
b-a(n-1)<0.
\]

Let $\chi\in C_c^\infty(U_2)$ be a cutoff function satisfying
$
0\leq\chi\leq1 $ and $
\chi\equiv1 $ on $U_1 $.
For $\tau>0$, define
\begin{equation}\label{eq:conformal-blow-up}
w:=1+\tau\chi\psi_{\mathcal S},
\qquad
g':=w^{2a}g,
\qquad
\rho:=b\log w.
\end{equation}

We first show that $g'$ is complete. It suffices to consider a
$g$-unit-speed curve
\[
\gamma\colon[0,L)\longrightarrow U_1\setminus\mathcal S
\]
such that $\gamma(t)\to\mathcal S$ as $t\to L$. Since
$
d_g(\gamma(t),\mathcal S)\leq L-t,
$
 and  \eqref{bound-prescribed-solution} gives
\[
w(\gamma(t))
\geq
C(L-t)^{-(1-\eta)}.
\]
It follows that
\[
L_{g'}(\gamma)
=
\int_0^Lw(\gamma(t))^a\,dt
\geq
C\int_0^L(L-t)^{-a(1-\eta)}\,dt
=\infty,
\]
where we have used $a(1-\eta)>1$. Thus $g'$ is complete.

We next estimate its weighted scalar curvature. The conformal
transformation formula, together with the choice of $a$ and $b$,
yields
\begin{equation}\label{eq:weighted-conformal-blow-up}
w^{2a}R_{\mu,\rho}(g')
=
R(g)
-
2\bigl(b-a(n-1)\bigr)\tau
\frac{-\Delta_g(\chi\psi_{\mathcal S})}
     {1+\tau\chi\psi_{\mathcal S}}.
\end{equation}
On $U_1$, where $\chi\equiv1$, 
using \eqref{bound-prescribed-solution} and \eqref{Laplace-estimate}, we obtain
$
w^{2a}R_{\mu,\rho}(g')
\geq
C\,d_g(x,\mathcal S)^{-2}
$ on $U_1\setminus\mathcal S
$.
Therefore,
\begin{equation}\label{eq:weighted-lower-near-S}
R_{\mu,\rho}(g')
\geq
Cw^{-2a}d_g(x,\mathcal S)^{-2}
\qquad\text{on }U_1\setminus\mathcal S.
\end{equation}
On the compact set $\overline{M\setminus U_1}$, define
$
C_1:=
\sup_{\overline{M\setminus U_1}}
\left|\Delta_g(\chi\psi_{\mathcal S})\right|<\infty.
$
Given $\varepsilon>0$, choose $\tau>0$ sufficiently small that
\[
2\tau\bigl(b+a(n-1)\bigr)C_1\leq\varepsilon.
\]
Using $w\geq1$ and   \eqref{eq:weighted-conformal-blow-up} gives
\begin{equation}\label{eq:weighted-lower-away-S}
    w^{2a}R_{\mu,\rho}(g')
\geq
R(g)-\varepsilon
\qquad\text{on }M\setminus U_1.
\end{equation}

Finally, set
\[
u_\varepsilon:=w,\qquad
g'_\varepsilon:=u_\varepsilon^{2a}g,\qquad
\rho_\varepsilon:=b\log u_\varepsilon.
\]
The function $u_\varepsilon$ is identically $1$ outside
$\operatorname{supp}\chi$ and tends to $+\infty$ along
$\mathcal S$. The conclusion now follows from
\eqref{eq:weighted-lower-near-S} and
\eqref{eq:weighted-lower-away-S}.
\end{proof}

\section{Proofs of the extremality results for scalar curvature}
In this section, we establish scalar-curvature extremality results for $L^\infty$-manifolds. Theorem \ref{B} proves scalar-curvature extremality for $L^\infty$ metrics on tori, allowing general singular sets at the cost of an additional assumption on the fundamental group of the singular set. When the singular set is a simplicial complex, this additional assumption can be removed; see Theorem \ref{thm:LinfLlarull} and Proposition \ref{ric-vanish-torus}.

\subsection{General singular sets} We combine the conformal change  with Theorem \ref{thm:torus} to establish the following scalar curvature extremality for $L^\infty$-metrics on tori. 
 
\begin{theorem}\label{B}Let $(\mathbb{T}^n, g)$ be a  $L^\infty$-manifold. Suppose the singular set $\mathcal{S}$ satisfies the Hausdorff dimension upper bound
\begin{equation*}
\dim_{\mathcal{H}}\mathcal{S}\leq n-3+\frac{2}{n}
\end{equation*} and the induced map $\pi_1(U)\rightarrow \pi_1(\tor^n)$ is not  surjective for some neighborhood $U$ of $\mathcal{S}$. If the scalar curvature $R(g)$ is non-negative on $\tor^n\setminus \mathcal{S}$, then $g$ is Ricci-flat on $M\setminus\mathcal{S}$.  
Moreover, if in addition $\mathcal{S}$ satisfies the Minkowski dimension upper bound
\begin{equation}
    \dim_{\mathcal{M}}\mathcal{S}\leq n-3+\frac{1}{n-1},
\end{equation} 
then there exist a flat metric $g_0$ on $\tor^n$ and a homeomorphism
\[
    \Phi:\tor^n\longrightarrow\tor^n
\]
such that the restriction
\[
    \Phi|_{\tor^n\setminus\mathcal S}:
    (\tor^n\setminus\mathcal S,g)
    \longrightarrow
    (\tor^n\setminus\Phi(\mathcal S),g_0)
\]
is a smooth Riemannian isometry. Furthermore, the metric completion of the length space
$(\tor^n\setminus\mathcal S,d_g)$ is isometric to the flat torus $(\tor^n,d_{g_0})$.
\end{theorem}

\begin{proof} We first prove the Ricci-flatness assertion. Suppose to the contrary that $\Ric_g$ does not vanish at some point of $\tor^n\setminus \mathcal{S}$. By Theorem \ref{dichotomy}, we may assume that the $L^\infty$-metric $g$ satisfies 
\begin{equation}
    R(g)>0 \text{ in } \tor^n\setminus \mathcal{S}.
\end{equation} Using Proposition \ref{prop:weightedSc}, there is a complete metric $g'$ on $\tor^n\setminus \mathcal{S}$, a constant $\mu>\frac{n-1}{n}$ and a smooth function $f\in C^\infty(\tor^n\setminus \mathcal{S})$ such that 
\begin{equation}
    R_{\mu, f}(g')>0
\end{equation}
Since $\pi_1(U)\rightarrow\pi_1(\tor^n)$ is not a surjection, this leads to a contradiction with Theorem \ref{thm:torus}. Hence
$\Ric_g=0$ on $\tor^n\setminus\mathcal S$.

The remaining rigidity assertions are proved in Section \ref{sec:rigidity}.
\end{proof}

\subsection{Simplical sucomplex singularities} 
For general singular sets, Theorem \ref{B} requires an additional assumption concerning the fundamental group of a neighborhood of the singular set. Consequently, it does not fully resolve Conjecture \ref{conj:Schoen} in the case of tori. We remedy this when $\mathcal S$ is an embedded simplicial subcomplex of dimension at most $n-3$, showing that the additional fundamental-group assumption can be removed. Recall that Conjecture \ref{conj:Schoen} concerns the case in which $\mathcal S$ is an embedded submanifold, which admits a compatible triangulation as an embedded simplicial complex. Thus, our result resolves Conjecture \ref{conj:Schoen} for tori. A crucial ingredient in our argument is the quantitative scalar-curvature estimate established in Proposition \ref{simplical-blow-up}. The same approach also gives the following version of Llarull's theorem for $L^\infty$ metrics.

\begin{theorem}\label{thm:LinfLlarull}
    Let $(M^n, g)$ be a closed spin $L^\infty$-manifold. Suppose that the singular set $\Sing\subset M$ a Lipschitz embedded simplicial complex with dimension at most $n-3$. If $\Phi\colon (M,g)\to (\sph^n,g_{\sph^n})$ is a Lipschitz map of non-zero degree that is smooth away from $\Sing$ and satisfies 
    $$R(g)-2|\wedge^2 d\Phi|^{g}_{\tr}\geq 0$$
 on $M\setminus\Sing$, then $\Phi$ is a metric isometry up to scaling.
\end{theorem}

\begin{proof}[Proof of Theorem \ref{thm:LinfLlarull}] Suppose, to the contrary, that $\Phi$ is not a homothety (i.e. isometry up to scaling). By  Theorem \ref{dichotomy-Llarull}, we may assume that 
    \begin{equation}\label{uniform-po}
    R(g)-2|\wedge^2 d\Phi|^g_{\tr}\geq \kappa>0\end{equation}
Proposition \ref{simplical-blow-up} yields constants 
$
    \mu>\frac{n-1}{n}$, $ a>0 $, $ C>0
$ and 
functions $u, f\in C^\infty(M^n\setminus \mathcal{S})$ and a neighborhood $U_1$ of $\mathcal{S}$ such that
the conformal metric $g'=u^{2a}g$ is complete on $M^n \setminus \mathcal{S}$ and  the scalar curvature of $g'$ satisfies 
    \begin{equation}\label{curvature-estimate-blow-up-2}\begin{split}
        &R_{\mu, f}(g')= u^{-2a} (R(g)-\kappa/2)\text{ on } M\setminus U_1\\
        &R_{\mu, f}(g')\geq Cu^{-2a}d_g(x, \mathcal{S})^{-2} \text{ on } U_1
        \end{split}
    \end{equation}

In order to apply Theorem \ref{thm:Llarull}, we now deform $\Phi$ such that it is locally constant near $\mathcal{S}$. Choose a regular value point $q\in \sph^n$ with $q\notin \Phi(\mathcal{S})$. Since $\mathcal{S}$ is compact, there exists $r_0$ and relative compact neighborhood $V\subset U_1$ of $\mathcal{S}$ such that 
\[\Phi(V)\subset \sph^n\setminus B(q, r_0).\]
Let $q^*$ be the antipodal point of $q$. In the geodesic polar coordinates $(r, \theta)$ centered at  $q$, define a  homotopy ${H}: [0, \pi]\times \mathbb{S}^n\longrightarrow \mathbb{S}^n$
\begin{equation}
{H} (t, r, \theta)=\left(\min\big\{(1-\frac{t}{\pi}+\frac{t}{r_0})r, \pi\big\}, \theta\right)
\end{equation}
Thus,  ${H}(0, \cdot)=id_{\sph^n}$, while ${H}(\pi, \sph^n\setminus B(q,r_0))=q^*$. Moreover, ${H}$ is Lipschitz and satisfies  
\begin{equation}\label{homotopy-boundness-2}
    |\nabla {H}|_{g_{\mathbb{S}^n}}(t, x)\leq C_{H}. 
\end{equation} for some constant $C_{H}<\infty$. For each $t\in [0, \pi]$, the map ${H}(t, \cdot)$ is homotopic to the identity and hence has degree one. 

\vspace{2mm}

Since $M^n$ is closed, there is a constant $L_\Phi$ with 
$
    |\nabla \Phi|\leq L_\Phi
$.
Choose $s_0<s_1$ sufficiently large that $\{d_{\mathcal{S}}(x)\leq e^{-s_0}\}\subset V$ and 
\begin{equation}\label{constant-choice-2}
    s_1-s_0\geq 4n(n-1)\pi C^{-1/2} C_{H}, \quad e^{s_0}\geq
			2n(n-1) C^{-1/2} C_HL_\Phi,
\end{equation}
where $d_{\mathcal S}(x)$ is the distance from $x$ to $\mathcal S$.
Choose a cutoff function $\varphi$ such that 
\begin{equation}\label{cut-off-choice}
    \varphi(s)=0 \text{ for } s\leq s_0, \quad \varphi(s)=\pi \text{ for } s\geq s_1, \quad 0\leq \varphi'(s)\leq \frac{C^{\frac 1 2}C^{-1}_{H} }{2n(n-1)} 
\end{equation}  and define a map $\Psi\colon  M\setminus \mathcal{S}\rightarrow \sph^n$
\begin{equation}
    \Psi(x):={H}(\varphi(-\log d_\mathcal{S}(x)), \Phi(x)) \quad x\in M\setminus \mathcal{S}.
\end{equation}
 Since $\varphi(-\log d_{{\mathcal{S}}}(x))=0$ for $x\in M\setminus V$, we have that $\Psi(x)=\Phi(x)$ outside $ V$. On the other hand, $\varphi(-\log d_{{\mathcal{S}}}(x))=\pi$ for $x$ sufficiently close to $\mathcal{S}$, which implies  \[\Psi(x)=q^*\] in a neighborhood of $\mathcal{S}$. After the extending this constant value across ${\mathcal{S}}$, the map  $\Psi$ is locally constant near ${\mathcal{S}}$. Furthermore, $\Psi$ is locally constant outside a compact set of $M\setminus {\mathcal{S}}$ and the homotopy $H$ implies that 
\begin{equation}\label{degree-psi2}
    \deg(\Psi)=\deg(\Phi)\neq 0. 
\end{equation}

It remains to show the positivity of $R_{\mu,f} (g')-2|\wedge^2\Psi|^{g'}_\tr$. Consider  first the region $\{e^{-s_1}\leq d_{{\mathcal{S}}}(x)\leq e^{-s_0}\}$.  On this region, 
the chain rule, \eqref{homotopy-boundness-2} and \eqref{cut-off-choice} give
\begin{equation}\label{grad-psi-2}
    |d\Psi|_{{g}'}\leq u^{-a}\cdot C_{H}\big(\frac{1}{2n(n-1)}C^{\frac{1}{2}}C^{-1}_{H}\cdot d^{-1}_{\mathcal {S}}(x)+L_\Phi \big)\leq  \frac{C^{\frac 1 2}}{n(n-1)} u^{-a} d^{-1}_{\mathcal {S}}(x)
\end{equation}
where the second inequality follows from \eqref{constant-choice-2} that 
\[ L_\Phi\leq \frac{1}{2n(n-1)}C^{-1}_{{H}}C^{\frac 1 2 }e^{s_1}\leq \frac{1}{2n(n-1)}C^{-1}_{{H}}C^{\frac 1 2 } d^{-1}_{\mathcal{S}}(x)
\] 
on this region. It follows from \eqref{curvature-estimate-blow-up-2} and \eqref{grad-psi-2} that on this region 
\begin{equation}
\begin{split}
   R_{\mu, f}(g') -|\wedge^2 d\Psi|^{\tilde{g}'}_\tr&\geq R_{\mu, f}(g')- n(n-1)|d\Psi|^2_{{g}'}\\
   &\geq u^{-2a}\left(
    {C}\cdot d_{{\mathcal{S}}}(x)^{-2}-\frac{C}{n(n-1)}\cdot d_{{\mathcal{S}}}(x)^{-2}\right) > 0
    \end{split}
\end{equation}
On the region $\{d_{\mathcal S}\leq e^{-s_1}\}$, the map $\Psi$ is
	constant. Therefore,
	\[
	R_{\mu,f}(g')
	-2\lvert\wedge^2d\Psi\rvert_{\tr}^{g'}
	=
	R_{\mu,f}(g')
	\geq
	C u^{-2a}d_{\mathcal S}^{-2}>0.
	\]
Finally, on the region $\{d_{\mathcal S}\geq e^{-s_0}\}$, we have
	$\Psi=\Phi$. Since $g'=u^{2a}g$,
	\[
	\lvert\wedge^2d\Phi\rvert_{\tr}^{g'}
	=
	u^{-2a}\lvert\wedge^2d\Phi\rvert_{\tr}^{g}.
	\]
	It follows from \eqref{curvature-estimate-blow-up-2} and
	\eqref{uniform-po} that
	\begin{align}
		R_{\mu,f}(g')
		-2\lvert\wedge^2d\Psi\rvert_{\tr}^{g'}
		&\geq
		u^{-2a}
		\left(
		R(g)-\frac{\kappa}{2}
		-2\lvert\wedge^2d\Phi\rvert_{\tr}^{g}
		\right)\notag\\
		&\geq
		\frac{\kappa}{2}u^{-2a}>0.
	\end{align}

Thus, $(M\setminus \mathcal{S}, g')$ is a complete spin manifold and $\Psi: M\setminus \mathcal{S}\rightarrow \sph^n$ is non-zero degree map that is locally constant outside a compact set satisfying 
\begin{equation}
    R_{\mu, f}(g')-2|\wedge^2 \Phi|^{g'}_\tr>0,
\end{equation}which leads to a contradiction to Theorem \ref{thm:Llarull}.
\end{proof}

Combining the proofs of Theorems \ref{thm:Llarull} and \ref{thm:LinfLlarull}, we have  the following immediate corollary.
\begin{corollary}\label{Corollary:noncompactLinfLlarull}
    Let $M^n$ be an open spin manifold and $\Sing\subset M$ a compact embedded simplicial complex with codimension at least three. Let $\Phi\colon M^n\to \sph^n$ be a Lipschitz map that is smooth away from $\Sing$ and is locally constant outside  some compact set , and  $\Phi$ has a non-zero degree. Then there is no complete $L^\infty$-metric $g$ on $M$ which  is smooth away from $\Sing$ and satisfies 
    $$R(g)-2|\wedge^2 d\Phi|^g_\tr>0.$$
\end{corollary}

We now apply Corollary \ref{Corollary:noncompactLinfLlarull} to the rigidity problem for enlargeable manifolds. The following proposition shows that, when the singular set is an embedded simplicial subcomplex of codimension at least three, nonnegative scalar curvature forces Ricci-flatness on the regular part.

\begin{proposition}\label{ric-vanish-torus} Let $(M^n, g)$ be a closed enlargeable and spin $L^\infty$-manifold. Suppose that the singular set $\mathcal{S}$ is an embedded simplicial subcomplex of dimension at most $n-3$. If the scalar curvature $R(g)$ is non-negative on $M^n\setminus \mathcal{S}$, then $\Ric_g$ vanishes on  $M^n\setminus \mathcal{S}$. 
\end{proposition}

\begin{proof} Suppose, to the contrary, that $\Ric_g$ does not vanish on $M^n\setminus \mathcal{S}$. From Theorem \ref{dichotomy}, we may assume that the $L^\infty$ metric $g$ satisfies 
\begin{equation}\label{curvatue-est}
    R(g)>\kappa>0 \text{ on } M^n\setminus \mathcal{S}.
\end{equation}
Fix $\varepsilon<\sqrt{\frac{\kappa}{2n(n-1)}}$.
By  enlargeability of $M$, there exists a spin cover map $\pi_\varepsilon: M_\varepsilon\rightarrow M$ and a smooth map $\Phi_\varepsilon: (M_\varepsilon, \pi^*_\varepsilon g)\rightarrow (\mathbb{S}^n, g_{\mathbb{S}^n})$ satisfying 
\begin{equation}\label{first-der-phi}
    \deg(\Phi_\varepsilon)\neq 0 \quad | d\Phi_\varepsilon|_{\pi^*_\varepsilon (g)}\leq \varepsilon. 
\end{equation} and it is locally constant outside a compact set $K$. 

The metric $\pi^*_\varepsilon{g}$ on $M_\varepsilon$ is complete  and \eqref{curvatue-est} remains valid after lifting to the covering space. It follows that 
\begin{equation}\begin{split}
    R(\pi^*_\varepsilon(g))-2|\wedge^2 d\Phi|^{\pi^*_\varepsilon(g)}_\tr&\geq R(\pi^*_\varepsilon(g))-n(n-1)|d\Phi_\varepsilon|^2_{\pi^*_\varepsilon(g)}\\
    &\geq \kappa-n(n-1)\varepsilon^2>0,
    \end{split}
\end{equation} which leads to a contradiction with Corollary \ref{Corollary:noncompactLinfLlarull}. 
\end{proof}

\section{Rigidity results for $L^\infty$-metrics}\label{sec:rigidity} In this section, we prove the rigidity assertions in Theorem \ref{B}. As an application, we give an affirmative answer to Conjecture \ref{conj:Schoen} for tori.
Throughout this section, we assume that $(\mathbb T^n,g)$ is an $L^\infty$-Riemannian manifold and $\mathcal S$ is its singular set such that 
\begin{equation}
{\dim}_{\mathcal M}(\mathcal S)
\leq n-3+\frac{1}{n-1} \textup{ and }
\operatorname{Ric}_g\equiv0
\text{ on }\mathbb T^n\setminus\mathcal S.
\end{equation}
We construct harmonic maps
$
u_\ell\colon(\mathbb T^n,g)\longrightarrow\mathbb S^1
$, $ 1\leq\ell\leq n,
$ and prove that their differentials are parallel and form a global frame of the cotangent bundle. The resulting harmonic coordinate map
$
(u_1,\ldots,u_n)\colon\mathbb T^n\longrightarrow\mathbb T^n
$ then provides a global coordinate system. These coordinates allow us to extend the metric structure across $\mathcal S$, thereby establishing the desired rigidity.

\subsection{Construction of Harmonic maps}\label{sec:harmonicMaps} Recall the $n$-torus is the quotient $\tor^n=\R^n/\Z^n$ and let
$\{e_\ell\}_{\ell=1}^n$ denote the standard basis of $\R^n$. The associated
coordinate functions on the universal cover determine the cohomology classes
from which the desired harmonic maps are constructed.

Fix the standard flat metric $g_0$ on $\tor^n$. Since $g$ is an
$L^\infty$ metric, there exists $\Lambda\geq 1$ such that
\begin{equation}\label{eq:uniform-equivalence}
	\Lambda^{-1}g_0\leq g\leq \Lambda g_0.
\end{equation}

\begin{lemma}\label{lemma:harmonicCordinates} 
There exist functions $\{u_\ell\}_{1\leq \ell\leq n}$ such that
	\begin{enumerate}
		\item $u_\ell \in W^{1,2}_{loc}(\R^n)$ and  $u_\ell(p+e_j)=u_\ell(p)+\delta_{\ell j}$, and
		\item for any $\varphi\in C_c^\infty(\R^n,g_0)$, we have $\int_{\R^n}\langle\nabla u_\ell,\nabla\varphi\rangle=0$ with respect to the metric $g$.
	\end{enumerate}
\end{lemma}
\begin{proof}
	Let $x_\ell$ be the coordinate function in the $\ell$-th direction. In particular, $x_\ell\in C^\infty(\R^n)$ and satisfies
	$
		x_\ell(p+e_j)= x_\ell(p)+\delta_{\ell j}
$
	for any $p\in\R^n$. Therefore, the gradient $\nabla x_\ell$ is invariant under the $\Z^n$-action, descends to $\tor^n$, and lies in $L^2(T(\tor^n))$. Consider the space
	$$W^{1,2}(\tor^n)_*=\left\{h\in W^{1,2}(\tor^n):\int_{\tor^n} h=0\right\}.$$
	Then there exists a constant $\lambda_1>0$ such that
	$$\|\nabla h\|^2\leq \|h\|_{W^{1,2}}^2=\|h\|^2+\|\nabla h\|^2\leq \left(1+\frac{1}{\lambda_1}\right)\|\nabla h\|^2.$$
	Thus, the map 
	$\nabla\colon W^{1,2}(\tor^n)_*\to L^2(T(\tor^n)$
	is injective and has closed range. From the closed graph theorem, $\nabla$ is invertible. 
    
    Let $P$ denote the orthogonal projection from $L^2(T(\tor^n))$ onto the image $\nabla(W^{1,2}(\tor^n)_*)$. Set
	\begin{equation}
		f_\ell=\nabla^{-1}(P(\nabla x_\ell)) \quad 
		u_\ell=x_\ell-\widetilde f_\ell,
	\end{equation}
	where $\widetilde f_\ell$ is the lift of $f_\ell$ to $\R^n$. Clearly, $u_\ell$ satisfies conditions (1). 
    
    It remain to verify item (2) in the statement. For our construction, the projection ensures that $(\nabla x_\ell-\nabla f_\ell)\perp \nabla(W^{1,2}(\tor^n)_*)$. For any $\psi\in C^\infty(\tor^n,g_0)$, the function $\psi-\langle\psi,1\rangle\cdot 1$ belongs to $W^{1,2}(\tor^n)_*$. Thus,	
	\begin{equation}
	    0=\int_{\tor^n}\langle \nabla x_\ell-\nabla f_\ell,\nabla(\psi-\langle\psi,1\rangle\cdot 1)\rangle=\int_{\tor^n}\langle\nabla u_\ell,\nabla\psi\rangle.
	\end{equation}
	The general case follows by decomposing any test function $\varphi\in C_c^\infty(\R^n,g_0)$ into a finite sum of lifts of functions from $\tor^n$.
\end{proof}

The following lemma states some regularity properties of these maps.

\begin{lemma}\label{lemma:estimate}
    Assume that $g$ is smooth on $\tor^n\setminus\Sing$. Denote by $\widetilde\Sing$ the lift of $\Sing$ to $\R^n$.
	The functions $\{u_\ell\}_{1\leq \ell\leq n}$ constructed in Lemma \ref{lemma:harmonicCordinates} satisfy:
	\begin{enumerate}
		\item $u_\ell$ is smooth away from $\widetilde{\mathcal S}$,
		\item $u_\ell\in C^{0,\alpha}(\R^n)$ for some $\alpha\in(0,1)$, and
		\item $|\nabla u_\ell|\leq C r^{-1+\alpha}$, where $r=\dist(~\cdot~,\widetilde{\mathcal S})$.
	\end{enumerate}
\end{lemma}
\begin{proof} Since $(\tor^n,g)$ is an $L^\infty$-manifold, the operator $\Delta_g$ may be written as a divergence-form elliptic operator with $L^\infty$ coefficients that are uniformly elliptic.  Standard De Giorgi--Nash--Moser theory \cite[Theorem 8.9]{Han} applies, even across the singular set, to yield $\alpha\in(0,1)$ such that $u_\ell \in C^{0,\alpha}(\tor^n)$.

We now prove the desired gradient estimate. Because the metric $g$ is smooth outside $\mathcal{S}$, standard gradient estimates for harmonic functions give a uniform bound for $|\nabla u_\ell|_g$ outside of $N_{r_0}(\mathcal{S})$, the closed $r_0$-neighborhood of $\mathcal{S}$ with respect to the background metric $g_0$, where $r_0>0$ is fixed. Take $p\in N_{r_0}(\mathcal{S})\setminus \mathcal{S}$, and set $r_p:=d_{g_0}(p,\mathcal{S})>0$. Then $B_{r_p /2}(p)\subset M^n \setminus\mathcal{S}$. Since the Ricci curvature vanishes in this ball, we can apply the 
%Consider the positive harmonic function on $B(x, r_x/2)$
% \[y^i-\inf_{B(x, r_x/2)} y^i+\epsilon\]
% for any positive constant $\epsilon>0$. Since $\text{Ric}_g$ vanishes on this ball, we apply 
Cheng--Yau gradient estimate \cite[Corollary 3.2, Chapter I]{SchoenYau1994Lectures} to obtain 
\begin{equation}\label{fahjg09haq09hh}
    |\nabla u_\ell(p)|_g\leq \sup_{B_{r_p /4}}|\nabla(u_\ell^i-\inf_{B_{r_p /2}(p)} u_\ell)|_g\leq cr_p^{-1} \sup_{B_{r_p /2}(p)}|u_\ell-\inf_{B_{r_p /2}(p)} u_\ell|.  %\leq cc_0 r_p^{\gamma-1},
\end{equation} 
Furthermore, by the H\"{o}lder estimate above
\begin{equation}\label{-09}
\sup\limits_{p', p''\in B_{r_p /2}(p)}|u_\ell(p')-u_\ell(p'')|\leq c_0 d_{g_0}(p',p'')^\alpha\leq c_1 r_p^{\alpha}. 
\end{equation}
Combining \eqref{fahjg09haq09hh} and \eqref{-09} yields the desired result. We note that the final constant obtained is uniform,
since $N_{r_0}(\mathcal{S})$ is compact.
\end{proof}

\subsection{Proof of the rigidity statement of Theorem \ref{B}} We begin with a family of cutoff functions adapted to the singular set.
\begin{lemma}\label{lemma:cut-off}
    Assume that $\Sing\subset \tor^n$ has Minkowski dimension at most $n-k$. For any  $\varepsilon>0$ and  any small $\delta>0$, there exists a cut-off function $\chi_\delta$ such that
$$0\leq \chi_\delta\leq 1,\quad
\chi_\delta=1 \text{ on } N_\delta(\Sing),  \quad \supp\, \chi_\delta\subset N_{2\delta}(\Sing)$$
\[ 
 \quad 
\int_{\tor^n} |\nabla \chi_\delta|^2dV_g\leq C_{\varepsilon}\cdot \delta^{k-2-\varepsilon}. \]
\end{lemma}
\begin{proof} There is a constant $C$ such that 
\[C^{-1}g_0\leq g\leq Cg_0\]
Since the Minkowski dimension of $\Sing$ $\leq n-k$, for every small $\varepsilon_1>0$ there exists a constant $C_{\varepsilon_1}'$ such that, for all
sufficiently small $\delta>0$,
\begin{equation}
\operatorname{Vol}_{g_0}\bigl(N_\delta (\Sing)\bigr)
\leq {C_{\varepsilon_1}'} \delta^{k-\varepsilon_1}.
\end{equation}
Here $N_\delta({\Sing})$ denotes the $\delta$-neighborhood of ${\Sing}$ with respect to $g_0$.

Choose smooth cutoff functions $\chi_\delta$ satisfying
\begin{equation}
0\leq \chi_\delta\leq 1,\quad
\chi_\delta=1 \text{ on } N_\delta(\Sing),
\quad
\supp\, \chi_\delta\subset N_{2\delta}(\Sing),
\quad 
|\nabla \chi_\delta|_{g_0}\leq \frac{2}{\delta}.
\end{equation}
Furthermore, it follows that 
$
|\nabla \chi_\delta|_g^2\,dV_g
\leq
C^{n+2} |\nabla \chi_\delta|_{g_0}^2\,dV_{g_0}.
$
Therefore
\begin{equation}
\int_{\tor^n} |\nabla \chi_\delta|_g^2\,dV_g
\leq
\frac{4C^{n+2}}{r^2}
\operatorname{Vol}_{g_0}\bigl(N_{2\delta}({S})\bigr)
\leq C_{\varepsilon_1}\delta^{k-2-\varepsilon_1}.
\end{equation}
\end{proof}

\begin{proposition}\label{prop:harmonicMap}
	Assume that $\textup{Ric}_g\geq 0$ on $\tor^n\setminus \Sing$, and that $\Sing$ has Minkowski dimension at most $n-3+(n-1)^{-1}$.
	Then the functions $u_\ell$ above have vanishing Hessians and are therefore Lipschitz. Furthermore, there exists a smooth flat metric $\overline{g}$ on $\R^n$ such that the induced map
	$$\overbar U=(\bar u_1,\ldots,\bar u_n)\colon \tor^n\to (\tor^n=\R^n/\Z^n,\overline{g})$$
	is a homeomorphism and satisfies $g=\overbar U^*\overline g$.
\end{proposition}
\begin{proof}
	Fix $\ell$ and, for simplicity, write $u=u_\ell$. On $\R^n\setminus\widetilde \Sing$, using the metric $g$, we have the Bochner formula:
	\begin{equation}\label{eq:Bochner}
		\frac 1 2\Delta(|\nabla u|^2)=\langle\nabla(\Delta u),\nabla u\rangle+|\textup{Hess}(u)|^2+\textup{Ric}_g(\nabla u,\nabla u)\geq |\textup{Hess}(u)|^2.
	\end{equation}
	Since $\nabla u$ descends to $\tor^n$, inequality \eqref{eq:Bochner} also holds on $(\tor^n\setminus \Sing,g)$. We define
	$$v=(|\nabla u|^2+1)^p$$
	where  $p\in(0,1)$ such that
	\begin{equation}\label{eq:choice-p}
		\frac{n-2}{2(n-1)}
		<p<
		\min\left\{
		1,
		\frac{1 - \frac{1}{n-1}}{2(1-\alpha)}
		\right\}.
	\end{equation} Taking the gradient yields
	$$\nabla v=p(|\nabla u|^2+1)^{p-1}\nabla(|\nabla u|^2)=2p(|\nabla u|^2+1)^{p-1}|\nabla u|\cdot\nabla(|\nabla u|).$$
	Computing the Laplacian, we find
	\begin{align*}
		\Delta v=&~p(|\nabla u|^2+1)^{p-1}\Delta(|\nabla u|^2)+p(p-1)(|\nabla u|^2+1)^{p-2}|\nabla(|\nabla u|^2)|^2\\
		=&~p(|\nabla u|^2+1)^{p-1}\Delta(|\nabla u|^2)+4p(p-1)(|\nabla u|^2+1)^{p-2}|\nabla u|^2\cdot \big|\nabla|\nabla u|\big|^2\\
		\geq &~2p(|\nabla u|^2+1)^{p-1}\left(\frac 1 2\Delta(|\nabla u|^2)+2(p-1)\big|\nabla|\nabla u|\big|^2\right)\\
        \geq&~2p(|\nabla u|^2+1)^{p-1}\left(|\textup{Hess}(u)|^2+2(p-1)\big|\nabla|\nabla u|\big|^2\right).
	\end{align*}
	By the refined Kato inequality, we have
	\begin{equation}
		|\textup{Hess}(u)|^2\geq\frac{n}{n-1}\big|\nabla|\nabla u|\big|^2.
	\end{equation}
    Thus
 \begin{equation}
		\Delta v\geq 2p\left(1+2(p-1)\frac{n-1}{n}\right)(|\nabla u|^2+1)^{p-1}|\textup{Hess}(u)|^2.
	\end{equation}
	We require that the coefficient remains positive:
	\begin{equation}\label{eq:2p>}
		1+2(p-1)\frac{n-1}{n}>0,\quad \textup{namely}\quad 2p>\frac{n-2}{n-1}.
	\end{equation}
	
	Let $\chi_\delta$ be the cut-off function as in Lemma \ref{lemma:cut-off}, and $\psi_\delta\coloneqq 1-\chi_\delta$. Therefore
	\begin{equation}
		\int_{\tor^n}|\nabla\psi_\delta|^2\lesssim\delta^{k-2-\varepsilon_1}.
	\end{equation}
	Because $\psi_\delta\in W^{1,2}(\tor^n)$ and vanishes near $S$, integration by parts yields
	\begin{equation*}
		\int_{\tor^n} \psi_\delta^2\Delta v=-\int_{\tor^n}\langle\nabla(\psi_\delta^2),\nabla v\rangle\geq\int_{\tor^n}\psi_\delta^2\cdot 2p\left(1+2(p-1)\frac{n-1}{n}\right)(|\nabla u|^2+1)^{p-1}|\textup{Hess}(u)|^2.
	\end{equation*}
	Rearranging and applying the chain rule, we have
	\begin{align*}
		&\int_{\tor^n}\psi_\delta^2\cdot 2p\left(1+2(p-1)\frac{n-1}{n}\right)(|\nabla u|^2+1)^{p-1}|\textup{Hess}(u)|^2\\
		\leq &\int_{\tor^n}|\nabla(\psi_\delta^2)|\cdot |\nabla v|=4p\int_{\tor^n}(|\nabla u|^2+1)^{p-1}\cdot \psi_\delta\big|\nabla|\nabla u|\big|\cdot|\nabla u|\cdot |\nabla\psi_\delta|\\
		\leq &4p\left(\int_{\tor^n} (|\nabla u|^2+1)^{p-1}\psi_\delta^2|\textup{Hess}(u)|^2\right)^{\frac 1 2}\left(\int _{\tor^n}
		(|\nabla u|^2+1)^{p-1}|\nabla u|^2\cdot |\nabla\psi_\delta|^2
		\right)^{\frac 1 2},
	\end{align*}
	where the last inequality follows from the Cauchy--Schwarz inequality and the Kato inequality. 
    
    Let $\alpha\in (0,1)$ be the parameter in Lemma \ref{lemma:estimate} and choose $p$ such that 
    $$\alpha=1+2(p-1)\frac{n-1}{n}.$$It follows that 
    	\begin{align*}
		&~\alpha\int_{\tor^n}\psi_\delta^2\cdot (|\nabla u|^2+1)^{p-1}|\textup{Hess}(u)|^2\\
		\leq& ~\frac{\alpha}{2}\int_{\tor^n}\psi_\delta^2\cdot (|\nabla u|^2+1)^{p-1}|\textup{Hess}(u)|^2
        +\frac{2}{\alpha}\int _{\tor^n}
		(|\nabla u|^2+1)^{p-1}|\nabla u|^2\cdot |\nabla\psi_\delta|^2,
	\end{align*}
    Canceling the common factor, we arrive at
	\begin{align*}
		\int_{\tor^n}\psi_\delta^2\cdot(|\nabla u|^2+1)^{p-1}|\textup{Hess}(u)|^2\leq \frac{4}{\alpha^2}\int_{\tor^n}
		(|\nabla u|^2+1)^{p}\cdot |\nabla\psi_\delta|^2.
	\end{align*}
	Applying Lemma \ref{lemma:estimate} and Lemma \ref{lemma:cut-off}, we can bound the right-hand side by
	\begin{equation}
		\int_{\tor^n}(|\nabla u|^2+1)^{p}|\nabla\psi_\delta|^2\lesssim \delta^{(-2+2\alpha)p}\cdot\int_{\tor^n}|\nabla\psi_\delta|^2\lesssim\delta^{(-2+2\alpha)p+k-2-\varepsilon_1}.
	\end{equation}
Choose $p$ satifying \eqref{eq:2p>} and $p<\frac{k-2}{2(1-\alpha)}$. 
We may then choose $\varepsilon_1$ sufficiently small to achieve $$(-2+2\alpha)p+k-2-\varepsilon_1>0.$$
By sending $\delta\rightarrow 0$, it follows that 
	\begin{equation}
		\lim_{\delta\to 0}\int_{\tor^n}\psi_\delta^2(|\nabla u|^2+1)^{p-1}|\textup{Hess}(u)|^2=0.
	\end{equation}
	Consequently, $\textup{Hess}(u)\equiv 0$. It immediately follows that $\nabla(\nabla u)=0$ identically on $(\tor^n\setminus S,g)$. 
	
	Now, consider the Gram metric 
	$
		G=(\langle \nabla u_\ell,\nabla u_{\ell'}\rangle)_{1\leq \ell,\ell'\leq n}
	$
	 associated with the parallel vector fields $\{\nabla u_\ell\}_{1\leq \ell\leq n}$. Since  these fields are parallel, $G$ is a constant positive semi-definite matrix on $\tor^n\setminus S$. We claim that $G$ is positive definite. 
\vspace{2mm}
Suppose otherwise. Then there exist constants $a_1,\ldots,a_n$, not all zero, such that
	\begin{equation}
		\sum_{\ell=1}^n a_\ell\nabla u_\ell=0.
	\end{equation}
	Passing to the universal cover and using the same notation for the lifted functions  imply that
	\begin{equation}
		\sum_{\ell=1}^n a_\ell u_\ell=C
	\end{equation}
	for some constant $C$ on $\R^n$. On the other hand, the equivariant relation implies  
	\begin{equation}
		C = \sum_{\ell=1}^n a_\ell u_\ell(p+e_j) = \sum_{\ell=1}^n a_\ell (u_\ell(p)+\delta_{\ell j}) = \sum_{\ell=1}^n a_\ell u_\ell(p)+a_j = C+a_j.
	\end{equation}
	This implies that all $a_j=0$, a contradiction. Hence, $G$ is positive definite ($\det(G)>0$). 

    Equip the target torus $\mathbb R^n/\mathbb Z^n$ with the flat metric
    \[ \overline g
		=
		\sum_{\ell,m=1}^n
		(G^{-1})_{\ell m}\,dy_\ell\,dy_m. \]
The equivariance of the functions $u_\ell$ ensures that the map
$U=(u_1,\ldots,u_n)$ descends to
\[
\overline U=(\overline u_1,\ldots,\overline u_n)
\colon\mathbb T^n\longrightarrow
(\mathbb R^n/\mathbb Z^n,\overline g).
\]By construction,
$
\overline U^*\overline g=g$ on $\mathbb T^n\setminus \mathcal S
$, so $\overline U$ is a local isometry on the regular part.
Recall from Lemma \ref{lemma:harmonicCordinates} that
\[
u_\ell(x)=x_\ell-\widetilde f_\ell(x),
\]where $\widetilde f_\ell$ is $\mathbb Z^n$-periodic. Therefore, the straight-line homotopy
\[
U_t(x)=x-t\widetilde f(x),
\qquad 0\leq t\leq1,
\]descends to a homotopy from the identity to $\overline U$. In particular,
$\deg(\overline U)=1$. Once $\overline U$ is shown to be a local homeomorphism across $\mathcal S$, compactness implies that it is a covering map. Its degree then forces it to be one-sheeted, and hence a homeomorphism.
\end{proof}

\begin{proof}[Proof of the rigidity statement in Theorem \ref{B}]
	Under the hypotheses of the rigidity statement, the Ricci-flatness part of
	Theorem \ref{B}, proved above, gives
	\[
	\operatorname{Ric}_g\equiv0
	\qquad\text{on }\tor^n\setminus\mathcal S.
	\]
	Proposition \ref{prop:harmonicMap} therefore yields a flat torus
	$(\R^n/\Z^n,\overline g)$ and a homeomorphism
	\[
	\overline U\colon\tor^n\longrightarrow\R^n/\Z^n
	\]
	whose restriction to $\tor^n\setminus\mathcal S$ is a smooth Riemannian
	isometry onto the complement of $\overline U(\mathcal S)$. The same
	proposition identifies the metric completion of
	$(\tor^n\setminus\mathcal S,d_g)$ with
	$(\R^n/\Z^n,d_{\overline g})$. This proves all the rigidity assertions.
\end{proof}

\begin{proof}[Proof of Theorem \ref{A}]
	Let $g$ be an $L^\infty$ metric on $\tor^n$ with singular set
	$\mathcal S$, and suppose that $R(g)\geq0$ on
	$\tor^n\setminus\mathcal S$. By Proposition \ref{ric-vanish-torus}, $g$ must be Ricci-flat on $\tor^n\setminus \mathcal{S}$. Since the singular set $\mathcal{S}$ is an embedded simplical subcomplex of dimension at most $n-3$, its Minkowski dimension is at most $n-3\leq n-3+\frac{1}{n-1}$.  Proposition \ref{prop:harmonicMap} implies the map $\overbar U$ defines a smooth structure on the entire torus $\tor^n$.
\end{proof}

\section*{Appendix}
In this appendix, we give the proofs of Lemma \ref{pt-poission-eq} and Lemma \ref{decay-integ-estimate}. 
\begin{proof}[Proof of Lemma \ref{pt-poission-eq}]
Let $G$ be the Dirichlet Green function furnished by
Lemma \ref{exist-Green}, and define
\[
\psi_p(x)
:=
\int_{U_2}
G(x,y)d_g(y,p)^{-\alpha}\,dV_g(y).
\]
Since $\alpha<n$, the function $d_g(\,\cdot\,,p)^{-\alpha}$ is locally
integrable. Thus $\psi_p$ is well defined and satisfies
\eqref{poisson}.

It remains to prove \eqref{two-side-estimate}. By the Green-function
bounds in Lemma \ref{exist-Green}, it suffices to estimate
\[
I(x,p)
:=
\int_{U_2}
d_g(x,y)^{2-n}d_g(y,p)^{-\alpha}\,dV_g(y).
\]
Set $r=d_g(x,p)$. Decompose $U_2$ into
\[
B(p,r/2),\qquad B(x,r/2),
\qquad
U_2\setminus\bigl(B(p,r/2)\cup B(x,r/2)\bigr).
\]
On the first two regions, the volume-growth estimates give
\[
\int_{B(p,r/2)}
d_g(x,y)^{2-n}d_g(y,p)^{-\alpha}\,dV_g(y)
\leq Cr^{2-\alpha},~
\int_{B(x,r/2)}
d_g(x,y)^{2-n}d_g(y,p)^{-\alpha}\,dV_g(y)
\leq Cr^{2-\alpha}.
\]
Splitting the remaining region according to which of
$d_g(x,y)$ and $d_g(p,y)$ is larger, and then integrating in annuli,
gives
\[
\int_{U_2\setminus(B(p,r/2)\cup B(x,r/2))}
d_g(x,y)^{2-n}d_g(y,p)^{-\alpha}\,dV_g(y)
\leq
C\int_{r/2}^{\infty}s^{1-\alpha}\,ds
\leq Cr^{2-\alpha},
\]
where we have used $\alpha>2$. This proves the upper bound.

For the lower bound, restrict the integral to $B(p,r/4)$. On this ball,
$d_g(x,y)\asymp r$, and hence
\[
\begin{aligned}
I(x,p)
&\geq
C^{-1}r^{2-n}
\int_{B(p,r/4)}d_g(y,p)^{-\alpha}\,dV_g(y)\\
&\geq
C^{-1}r^{2-n}
\int_0^{r/4}s^{n-1-\alpha}\,ds
\geq C^{-1}r^{2-\alpha}.
\end{aligned}
\]
The desired estimate follows.
\end{proof}

\begin{proof}[Proof of Lemma \ref{decay-integ-estimate}]
Set $r=d_g(x,\sigma)$, and choose $p_x\in\sigma$ such that
$d_g(x,p_x)=r$. The Lipschitz simplex $\sigma$ is uniformly
$k$-Ahlfors regular at sufficiently small scales. In particular,
\[
\mathcal H^k\bigl(B(p_x,s)\cap\sigma\bigr)\asymp s^k.
\]
Since
$
r\leq d_g(x,y)\leq3r
$ and $
\text{for }y\in B(p_x,2r)\cap\sigma,
$
we obtain
$$
\int_{B(p_x,2r)\cap\sigma}
f(x,y)\,d\mathcal H^k(y)
\asymp r^{k-\beta}.
$$
This proves the lower bound and the corresponding local upper bound.

For the remaining region, decompose $\sigma\setminus B(p_x,2r)$ into
dyadic annuli centered at $p_x$. Since $\beta>k$,
\[
\begin{aligned}
\int_{\sigma\setminus B(p_x,2r)}
f(x,y)\,d\mathcal H^k(y)
&\leq
C\sum_{j\geq1}
(2^jr)^{-\beta}(2^jr)^k\\
&\leq
Cr^{k-\beta}
\sum_{j\geq1}2^{j(k-\beta)}
\leq Cr^{k-\beta}.
\end{aligned}
\]
This proves \eqref{eq:simplex-integral-estimate}.
\end{proof}
\bibliographystyle{plain}
\bibliography{ref}

\end{document}